\documentclass{amsart}
\usepackage{amsfonts}
\usepackage{bbm}
\usepackage{mathrsfs}
\usepackage{color}
\usepackage{amsmath}
\usepackage{bigdelim}
\usepackage{multirow}
\usepackage{amssymb}
\usepackage{indentfirst}
\usepackage{CJK}
\usepackage{fancyhdr}
\usepackage{amscd,amssymb,amsmath,graphicx,verbatim,xy}
\usepackage[TS1,OT1,T1]{fontenc}

\newtheorem{theorem}{Theorem}[section]
\newtheorem{lemma}[theorem]{Lemma}
\newtheorem{corollary}[theorem]{Corollary}
\newtheorem{proposition}[theorem]{Proposition}

\theoremstyle{definition}
\newtheorem{definition}[theorem]{Definition}
\newtheorem{example}[theorem]{Example}

\theoremstyle{remark}
\newtheorem{remark}[theorem]{Remark}
\numberwithin{equation}{section}

\usepackage{url}

\newcommand{\ii}{\boldsymbol{i}}
\newcommand{\jj}{\boldsymbol{j}}
\newcommand{\kk}{\boldsymbol{k}}

\makeatletter
\@namedef{subjclassname@2020}{%
	\textup{2020} Mathematics Subject Classification}
\makeatother

\begin{document}

\title[Complex Representatives, Similarity Invariants, and Iterates]{Complex Representatives, Similarity Invariants, and Iterates in $\bf{U(1,1;\mathbb{H})}$}

\author{Bingzhe Hou}
\address{Bingzhe Hou, School of Mathematics, Jilin University, 130012, Changchun, P. R. China}
\email{houbz@jlu.edu.cn}

\author{Cailing Yao}
\address{Cailing Yao, School of Medical Informatics, Changchun University of Chinese
	Medicine, 130117, Changchun, P. R. China} \email{yaocl@ccucm.edu.cn}
	
\author{Jiaxing He}
\address{School of Mathematics and Statistics, Changchun University of Science and Technology, 130022, Changchun, P. R. China}\email{547337872@qq.com}

\date{}
\begin{abstract}
We study the similarity classes of matrices in
$U(1,1;\mathbb H)$ and the iterates of their associated quaternionic
M\"obius transformations. Building on the known elliptic case, we show
that every parabolic or loxodromic element of $U(1,1;\mathbb H)$ is
similar, via a matrix in $U(1,1;\mathbb H)$, to a complex matrix. Explicit
complex representatives and matrices realizing these similarities are constructed for the
different types. We then introduce the invariants
$\operatorname{Tr}(T),
\Delta(T),
\eta(T),
\kappa(T)$,
and prove that they form a complete system of similarity invariants for
matrices in $U(1,1;\mathbb H)$. Finally, using the complex representation
of quaternionic matrices and Chebyshev polynomials, we derive a
fourth-order recurrence for a real sequence $(q_n)$ satisfying $\Delta(T^n)=q_n^2\Delta(T)$.
Together with recurrence relations for the relevant matrix entries, this
identity yields a complete classification of the quaternionic M\"obius
transformation associated with $T^n$, including the degenerate cases in
which $\Delta(T^n)=0$.
\end{abstract}
\maketitle

\section{Introduction and preliminaries}


The discovery of quaternions by the Irish mathematician W. Hamilton on October 16, 1843, marked a pivotal moment in the history of algebra and geometry. As a profound four-dimensional generalization of complex numbers, quaternions not only extend the real and imaginary components of complex numbers but also form a noncommutative associative division algebra, retaining associativity and distributivity while abandoning commutativity of multiplication. Today, quaternions are indispensable for the robust and efficient representation of three-dimensional rotations, with important applications in computer graphics and animation, particularly in camera movement and character rigging \cite{64,24}; robotics, for calculating spatial orientations \cite{23}; video games and virtual reality \cite{26}; and aerospace engineering \cite{22}. Beyond these practical applications, they also play an important role in theoretical physics, particularly in the mathematical formalism of quantum mechanics \cite{5}.

We first recall some basic definitions about quaternions. Let $\mathbb{H}$ denote the set of all quaternions. Elements of $\mathbb{H}$ have the form $z=z_{0}+z_{1}\boldsymbol{i}+z_{2}\boldsymbol{j}+z_{3}\boldsymbol{k}$ where $z_{i}\in $ $\mathbb{R}$ and $\boldsymbol{i}^{2}=\boldsymbol{j}^{2}=\boldsymbol{k}^{2}=\boldsymbol{ijk}=-1$. Let $\overline{z}=z_{0}-z_{1}\boldsymbol{i}-z_{2}\boldsymbol{j}-z_{3}\boldsymbol{k}$ be the conjugate of $z$, and define the modulus of $z$ to be:
\begin{center}
	$\mid z\mid=\sqrt{\overline{z}z}=\sqrt{z_{0}^{2}+z_{1}^{2}+z_{2}^{2}+z_{3}^{2}}$.
\end{center}
$\text{Re}(z)=z_{0}$ and $\text{Im}(z)=z_{1}\boldsymbol{i}+z_{2}\boldsymbol{j}+z_{3}\boldsymbol{k}$ denote the real and imaginary parts of $z$, respectively. Denote by $\mathbb{B}$ the unit ball of $\mathbb{H}$. That is, $\mathbb{B}=\{z\in\mathbb{H}:|z|^{2}<1\}$. Two quaternions $z$ and $\omega$ are similar if there exists a non-zero $q\in $ $\mathbb{H}$ such that $z=q\omega q^{-1}$. Quaternions $z$ and $\omega$ are similar if and only if $\mid z\mid=\mid \omega\mid$ and $z_{0}=\omega_{0}$. A quaternionic fractional linear transformation is of the form
$$
h(z)=(az+b)(cz+d)^{-1},\ a,b,c,d\in\mathbb{H}.
$$
Define the Trace of a $n\times n$ quaternionic matrix $A=[a_{ij}]$ \cite{DB2008} to be
$$
Tr(A)=\sum_{i=1}^{n}(a_{ii}+\overline{a_{ii}})=2\sum_{i=1}^{n}Re(a_{ii}).
$$
It is proved in \cite{WJX04} that the quaternionic M\"{o}bius transformation $g(z)=(az+b)(cz+d)^{-1}$ maps the unit ball $\mathbb{B}\subset \mathbb{H}$ to itself if and only if
\begin{equation}\label{1.1}
	\mid a\mid=\mid d\mid,\ \mid b\mid=\mid c\mid,\ \mid a\mid^{2}-\mid c\mid^{2}=1,\ \overline{a}b=\overline{c}d,\ a\overline{c}=b\overline{d}.
\end{equation}
Next we introduce $U(1,1;\mathbb{H})$ briefly, which is  called $Sp(1,1)$ sometimes, more details can be seen in  \cite{CWS2016}, \cite{WJX04}, \cite{IKJRP2003}.

Let $\mathbb{H}^{1,1}$ be the right vector space of dimension $2$ over $\mathbb{H}$ with the unitary structure defined by the Hermitian form
$$
\langle \boldsymbol{z},\boldsymbol{w} \rangle_{J}=\boldsymbol{w}^{\ast}J\boldsymbol{z}=\overline{w_{1}}z_{1}-\overline{w_{2}}z_{2},
$$
where $\boldsymbol{z}$ and $\boldsymbol{w}$ are the column vectors in $\mathbb{H}^{1,1}$ with entries $(z_{1},z_{2})^{T}$ and $(w_{1},w_{2})^{T}$ respectively, $\boldsymbol{w}^{\ast}$ denotes the conjugate transpose of $\boldsymbol{w}$ and $J$ is the Hermitian matrix $J=\begin{bmatrix}
	1 & 0 \\
	0 & -1
\end{bmatrix}$.
We define a $J$-unitary transformation $g$ to be an automorphism of $\mathbb{H}^{1,1}$, that is, a linear bijection such that $\langle g\boldsymbol{z},g\boldsymbol{w} \rangle=\langle \boldsymbol{z},\boldsymbol{w} \rangle$ for all $\boldsymbol{z}$ and $\boldsymbol{w}$ in $\mathbb{H}^{1,1}$. Denote the group of all unitary transformations by $U(1,1;\mathbb{H})$.

Let $T=\begin{bmatrix}
	a & b \\
	c & d
\end{bmatrix}\in U(1,1;\mathbb{H})$, the map
$$
\psi:U(1,1;\mathbb{H})\rightarrow M
$$
$$
\psi(T)=(az+b)(cz+d)^{-1}=g(z)
$$
is a group homomorphism from $U(1,1;\mathbb{H})$ to the group M of all quaternionic M\"{o}bius transformations. It is worth noticing that for each $T=\begin{bmatrix}
	a & b \\
	c & d
\end{bmatrix}\in U(1,1;\mathbb{H})$, the inverse of $T$ always has the form $T^{-1}=\begin{bmatrix}
	\overline{a} & -\overline{c} \\
	-\overline{b} & \overline{d}
\end{bmatrix}\in U(1,1;\mathbb{H})$ and $\det_{\mathbb{H}}(T)=1$, where  \cite{BC09}
$$
\text{det}_{\mathbb{H}}(T)=\sqrt{\mid a\mid^{2}\mid d\mid^{2}+\mid c\mid^{2}\mid b\mid^{2}-2Re(c\overline{a}b\overline{d}) }.
$$

It can be seen that for $T\in U(1,1;\mathbb{H})$, the equality $T^{\ast}JT=J$ always holds. Moreover, for $\boldsymbol{z}=(z_{1},z_{2})^{T}\in \mathbb{H}^{1,1}$,
$$
\boldsymbol{z}^{\ast}J\boldsymbol{z}=|z_{1}|^{2}-|z_{2}|^{2}.
$$
Consequently, a nonzero vector may also satisfy $\boldsymbol{z}^{\ast}J\boldsymbol{z}=0$. Such vectors are referred to as $J$-zero vectors or isotropic vectors. They are exactly the vectors that represent the boundary points of quaternionic hyperbolic lines.
It should be noted that
$$
\langle \boldsymbol{z}\alpha,\boldsymbol{w}\beta\rangle_{J}=\overline{\beta}\langle \boldsymbol{z},\boldsymbol{w}\rangle_{J}\alpha
$$
for any $\alpha,\beta\in\mathbb{H}$.

For $T=\begin{bmatrix}
	a & b \\
	c & d
\end{bmatrix}\in U(1,1;\mathbb{H})$, if $b\neq \overline{c}$ and $b\neq 0$, it follows from \cite{WJX04} that
\begin{equation}\label{1.2}
	\triangle=\mid \text{Im}(\overline{c}^{-1}b\overline{d}+d)\mid^{2}-\mid\overline{c}^{-1}b-1\mid^{2},
\end{equation}
$\triangle$ plays a crucial part in \cite{WJX04}. C. Yao et al. \cite{YHF} extended the definition of $\triangle$ to all cases and denoted the extension by  $\triangle(T)$:
$$
\triangle(T)=|b-\overline{c}|^{2}-(a_{0}-d_{0})^{2}.
$$
Denote by $M_{n}(\mathbb{H})$ the set of all $n\times n$ matrices of quaternions. We next present some related definitions for $A\in M_{n}(\mathbb{H})$.
\begin{definition}[\cite{Huang}]
	Let $A\in M_{n}(\mathbb{H})$. If there exists a quaternion $\lambda$ and a nonzero vector $\boldsymbol{v}$ such that
	$$
	A\boldsymbol{v}=\boldsymbol{v}\lambda,
	$$
	then $\lambda$ is a right eigenvalue of $A$ and $\boldsymbol{v}$ is a right eigenvector of $A$ corresponding to $\lambda$.
\end{definition}

Generalized eigenvalues and generalized eigenvectors are two widely used concepts in the theory of complex matrices \cite{we}. To address the similarity problem in the compound parabolic case, we hereby give the definition of generalized right eigenvectors for matrices in $U(1,1;\mathbb{H})$.

\begin{definition}\label{3.1}
	Suppose that $A\in U(1,1;\mathbb{H})$. Let $\boldsymbol{v}\neq \boldsymbol{0}$ be in $\mathbb{H}^{1,1}$ and $A\boldsymbol{v}=\boldsymbol{v}\lambda$ for some $\lambda\in\mathbb{H}$. If there exists a vector $\boldsymbol{w}\in\mathbb{H}^{1,1}$ which is right linearly independent of $\boldsymbol{v}$ such that
	$$
	A\boldsymbol{w}=\boldsymbol{w}\lambda+\boldsymbol{v},
	$$
	then we say $\boldsymbol{w}$ is a generalized right eigenvector of order $2$. More generally, let $\alpha\in\mathbb H\setminus\{0\}$ satisfy
	$\alpha\lambda=\lambda\alpha$. If
	\[
	A\boldsymbol{v}=\boldsymbol{v}\lambda,\qquad A\boldsymbol{w}=\boldsymbol{w}\lambda+\boldsymbol{v}\alpha,
	\]
	then we call $(\boldsymbol{v},\boldsymbol{w})$ a right $J$-chain with chain coefficient
	$\alpha$. It is called standard if $\alpha=1$. It is called
	$J$-normalized if
	\[
	\boldsymbol{v}^*J\boldsymbol{v}=\boldsymbol{w}^*J\boldsymbol{w}=0,\qquad \boldsymbol{v}^*J\boldsymbol{w}=1.
	\]
\end{definition}
For $A\in U(1,1;\mathbb{H})$, let $\lambda,\boldsymbol{v},\boldsymbol{w}$ be as defined in Definition \ref{3.1}. For a standard right $J$-chain, define a real-linear map
$$
\mathcal{N}_{\lambda}(\boldsymbol{x})=A\boldsymbol{x}-\boldsymbol{x}\lambda.
$$
Then $\mathcal{N}_{\lambda}(\boldsymbol{v})=\boldsymbol{0}$ and $\mathcal{N}_{\lambda}(\boldsymbol{w})=\boldsymbol{v}\neq \boldsymbol{0}$. Furthermore, we have
$$
\mathcal{N}^{2}_{\lambda}(\boldsymbol{w})=\mathcal{N}_{\lambda}(\boldsymbol{v})=\boldsymbol{0}.
$$
Thus, a nontrivial generalized right eigenvector exists if and only if
$$
\ker \mathcal{N}_{\lambda}\subsetneq \ker \mathcal{N}^{2}_{\lambda}.
$$

W. Cao et al. \cite{WJX04} classify M\"{o}bius transformations $g(z)=(az+b)(cz+d)^{-1}$ into six types.
\begin{enumerate}
	\item If $b=\overline{c}=0$, then $g$ is elliptic and 
	\begin{enumerate}
		\item if $a_{0}=d_{0}$, then $g$ is simple elliptic;
		\item if $a_{0}\neq d_{0}$, then $g$ is compound elliptic.
	\end{enumerate}
	\item If $b=\overline{c}\neq 0$, then
	\begin{enumerate}
		\item if $d_{0}^{2}<1$, then $g$ is simple elliptic;
		\item if $d_{0}^{2}=1$, then $g$ is simple parabolic;
		\item if $d_{0}^{2}>1$, then $g$ is simple loxodromic.
	\end{enumerate}
	\item If $b\neq\overline{c}$, then
	\begin{enumerate}
		\item if $\triangle<0$, then $g$ is compound elliptic;
		\item if $\triangle=0$, then $g$ is compound parabolic;
		\item if $\triangle>0$, then $g$ is compound loxodromic.
	\end{enumerate}
\end{enumerate}

They also classify the types of M\"{o}bius transformations $g$ by the fixed points of $T=\psi^{-1}(g)$ and right eigenvalues of $T$. 

A non-trivial quaternionic M\"{o}bius transformation $g$ preserving $\mathbb{B}$ is:
\begin{enumerate}
	\item simple elliptic if $T$ has right eigenvalues $t$ with $\mid t\mid=1$ and a two-dimensional family of fixed points;
	\item compound elliptic if $T$ has right eigenvalues $t$ with $\mid t\mid=1$ and a unique fixed point in $\mathbb{B}$;
	\item simple parabolic if $T$ has the repeated right eigenvalue $t=\pm 1$;
	\item compound parabolic if $T$ has a repeated right eigenvalue $t\neq\pm 1$ with $\mid t\mid=1$ and a fixed point on $\partial\mathbb{B}$;
	\item simple loxodromic if $T$ has the real right eigenvalues $t$ and $1/t$ where $t\neq \pm 1$;
	\item compound loxodromic if $T$ has non-real right eigenvalues $t_{1}$ and $t_{2}$ with $\mid t_{2}\mid=\mid t_{1}\mid^{-1}$ and $\mid t_{1}\mid\neq 1$.
\end{enumerate}

The conjugacy problem for quaternionic M\"obius transformations in the
full projective group was studied by Parker and Short
\cite{ParkerShort2009}. They obtained a conjugacy classification and, in
particular, observed that every quaternionic M\"obius transformation is
conjugate to one with complex coefficients. Accordingly, the existence
of a complex representative by itself is not claimed as new here. Our
problem is more restrictive: for \(T\in U(1,1;\mathbb H)\), we require
the conjugating matrix \(S\) itself to belong to \(U(1,1;\mathbb H)\), so
that
\[
STS^{-1}\in U(1,1;\mathbb H)\cap M_2(\mathbb C).
\]
Conjugacy in the full group does not imply this restricted statement,
since a general conjugator in \(SL(2,\mathbb H)\) need not satisfy
\(S^{*}JS=J\).

The elliptic case under this restricted notion of similarity was
established in \cite{YHF}. In Sections~2 and~3 we complete the
loxodromic and parabolic cases by constructing explicit \(J\)-unitary
conjugators and complex normal forms. The group constraint is essential
in the proofs: the loxodromic construction uses a \(J\)-normalization of
isotropic eigendirections, whereas the compound-parabolic construction
requires a \(J\)-normalized right Jordan chain. These constructions are
the additional content beyond the unrestricted existence of complex
representatives.

Quaternionic M\"obius iteration was already investigated by Heidrich
and Jank \cite{HeidrichJank1996}, who described the iterative behavior
of quaternionic M\"obius transformations and classified their main
conjugacy classes. Thus, our contribution is not the first treatment of
quaternionic iteration, but an explicit algebraic refinement for
elements of \(U(1,1;\mathbb H)\). More precisely, we derive the
palindromic fourth-order recurrence for \(T^n\) and for the real
sequence \(q_n\), prove
\[
\Delta(T^n)=q_n^2\Delta(T),
\]
and establish
\[
b_n-\overline{c_n}=q_n(b-\overline c),\qquad
\operatorname{Re}(a_n)-\operatorname{Re}(d_n)
=q_n\bigl(\operatorname{Re}(a)-\operatorname{Re}(d)\bigr).
\]
These identities lead to a six-type classification of \(g_{T^n}\).
In particular, when \(q_n=0\), and hence \(\Delta(T^n)=0\), the
additional conditions involving \(b_n\), \(c_n\), and
\(\operatorname{Tr}(T^n)\) distinguish the compound-parabolic,
simple-elliptic, simple-parabolic, and simple-loxodromic possibilities
that cannot be resolved from the sign of \(\Delta(T^n)\) alone.

This paper is organized as follows. In Section $2$ and Section $3$, we prove the complex representatives of loxodromic and parabolic elements in $U(1,1;\mathbb{H})$ separately. In Section $4$, we give the complete similarity invariants for matrices in $U(1,1;\mathbb{H})$. In Section $5$, we derive the recurrence formula of $\triangle(T^n)$ and completely classify the M\"{o}bius transformation $g_{T^n}$ corresponding to the matrix $T^n$.

\section{Complex representatives of loxodromic elements in $U(1,1;\mathbb{H})$}
In this section we prove that every matrix $T$ whose related quaternionic M\"{o}bius  transformation $\psi(T)$  is loxodromic is similar to a complex matrix in $U(1,1;\mathbb{H})$.

\begin{lemma}[\cite{FZZ1997}]\label{1}
	Two nonzero quaternions $p,q\in \mathbb{H}$ are similar if and only if their real parts and moduli are equal.
\end{lemma}
By taking advantage of $\triangle(T)$, a unified expression of $\sigma_{R}(T)$ was provided in \cite{YHF}.

\begin{lemma}[\cite{YHF}]\label{3}
	Let $T=\begin{bmatrix}
		a & b \\
		c & d
	\end{bmatrix}\in U(1,1;\mathbb{H})$. Then
	$$
	\sigma_{R}(T)=\{ s\in\mathbb{H}:s_{0}=\frac{1}{2}((a_{0}+d_{0})\pm\sqrt{2X^{\prime}-2-\triangle(T)}),\ \mid s\mid^{2}=X^{\prime}\pm\sqrt{X^{\prime 2}-1} \},
	$$
	where
	$$
	X^{\prime}=\frac{(a_{0}+d_{0})^{2}+\triangle(T)+\sqrt{((a_{0}+d_{0})^{2}+\triangle(T)-4)^{2}+16\triangle(T)\cdot \left(\frac{1}{2}sgn(\triangle(T))+\frac{1}{2}\right)}}{4}
	$$
	and $sgn$ is the sign function which means for any $x\in\mathbb{R}$,
	\begin{displaymath}
		sgn(x)=\left \{ \begin{array}{ll}
			1,\ x>0,\\
			0,\ x=0,\\
			-1,\ x<0.
		\end{array} \right.
	\end{displaymath}
	
\end{lemma}

\begin{theorem}\label{2}
	Let $g(z)=(az+b)(cz+d)^{-1}$ be a simple loxodromic quaternionic M\"{o}bius  transformation, $T=\begin{bmatrix}
		a & b \\
		c & d
	\end{bmatrix}$, then there exists a matrix $X \in U(1,1;\mathbb{H})$ such that $XTX^{-1}$ is complex.
\end{theorem}
\begin{proof}
	In this case, we have $b=\overline{c}\neq 0,\ d_{0}^{2}>1$ and  $T=\begin{bmatrix}
		a & \overline{c} \\
		c & d
	\end{bmatrix}$. Let
	$$
	M=\begin{bmatrix}
		\overline{q} & 0 \\
		0 & 1
	\end{bmatrix}\in U(1,1;\mathbb{H}),
	$$
	where $\overline{q}=\frac{c}{\mid c\mid}=q^{-1}=|c|(\overline{c})^{-1}$. Then
	$$
	MTM^{-1}=\begin{bmatrix}
		\overline{d} & |c| \\
		|c| & d
	\end{bmatrix}.$$
	Furthermore, by Lemma \ref{1}, we let $y_{1}$ be the unimodulus quaternion such that
	$$
	y_{1}\overline{d}\overline{y_{1}}=d_{0}+\ii\sqrt{1-d_{0}^{2}+|c|^{2}},
	$$
	this is equivalent to
	$$
	y_{1}d\overline{y_{1}}=d_{0}-\ii\sqrt{1-d_{0}^{2}+|c|^{2}}.
	$$
	If we let $N=\begin{bmatrix}
		y_{1} & 0 \\
		0 & y_{1}
	\end{bmatrix}\in U(1,1;\mathbb{H})$ and $X=NM$, then $X\in U(1,1;\mathbb{H})$ and
	$$
	XTX^{-1}=NMTM^{-1}N^{-1}=\begin{bmatrix}
		d_{0}+\ii\sqrt{1-d_{0}^{2}+|c|^{2}} & |c| \\
		|c| & d_{0}-\ii\sqrt{1-d_{0}^{2}+|c|^{2}}
	\end{bmatrix}.
	$$
	
	It is obvious that the matrix $XTX^{-1}$ is complex, which validates our conclusion.

\end{proof}
\begin{example}
	
	Let
	\[
	T=\frac14\begin{bmatrix}5&-3\boldsymbol{j}\\3\boldsymbol{j}&5\end{bmatrix}.
	\]
	Then $b=\bar c\ne0$ and $d_0^2=25/16>1$. Moreover,
	\[
	|a|^2-|c|^2=|d|^2-|b|^2=1,
	\qquad \bar a b-\bar c d=0,
	\]
	so $T\in U(1,1;\mathbb H)$ and the corresponding M\"obius transformation $g_T$ is
	simple loxodromic. Set
	\[
	X=\begin{bmatrix}\boldsymbol{j}&0\\0&1\end{bmatrix}\in U(1,1;\mathbb H).
	\]
	A direct calculation gives
	\[
	XTX^{-1}=\begin{bmatrix}5/4&3/4\\3/4&5/4\end{bmatrix}
	=\begin{bmatrix}\cosh(\log2)&\sinh(\log2)\\
		\sinh(\log2)&\cosh(\log2)\end{bmatrix}.
	\]
	Thus $T$ is similar to a real, hence complex, matrix in $U(1,1;\mathbb H)$.
\end{example}

\begin{theorem}\label{4}
	Let $g(z)=(az+b)(cz+d)^{-1}$ be a compound loxodromic quaternionic M\"{o}bius  transformation, $T=\begin{bmatrix}
		a & b \\
		c & d
	\end{bmatrix}$, then there exists a matrix $W \in U(1,1;\mathbb{H})$ such that $WTW^{-1}$ is  complex.
\end{theorem}
\begin{proof}
	In this case, we have $b\neq \overline{c},\  \triangle(T)=\triangle>0$. At this point, $T$ has two distinct boundary fixed points $z_{+}$ and $z_{-}$ on $\partial\mathbb{B}$, satisfying
	$$
	|z_{+}|=|z_{-}|=1,\ \ z_{+}\neq z_{-.}
	$$
	Assume
	$$
	u_{+}=\begin{bmatrix}
		z_{+} \\
		1
	\end{bmatrix},\ \ u_{-}=\begin{bmatrix}
		z_{-} \\
		1
	\end{bmatrix},
	$$
	and
	$$
	t_{+}=cz_{+}+d,\ \ t_{-}=cz_{-}+d.
	$$
	Therefore,
	$$
	Tu_{+}=u_{+}t_{+},\ \ Tu_{-}=u_{-}t_{-}.
	$$
	That is, $t_{+}$ and $t_{-}$ are representatives of the two similarity equivalence classes of the right eigenvalues of $T$, respectively. $u_{+}$ and $u_{-}$ are right eigenvectors corresponding to $t_{+}$ and $t_{-}$, respectively. In fact, we can obtain the explicit forms of $t_{+}$ and $t_{-}$ from Lemma \ref{3}. Define $h=\overline{z_{+}}z_{-}-1$, then $h\neq 0$. Set
	$$
	S=\begin{bmatrix}
		z_{+}&z_{-}h^{-1}\\
		1 & h^{-1}
	\end{bmatrix}=(u_{+},u_{-}h^{-1}),\ \ H_{0}=\begin{bmatrix}
		0&1\\
		1&0
	\end{bmatrix}.
	$$
	Then $S^{\ast}JS=H_{0}$ and
	$$
	S^{-1}=H_{0}S^{\ast}J=\begin{bmatrix}
		\overline{h^{-1}}\overline{z_{-}}&-\overline{h^{-1}}\\
		\overline{z_{+}} & -1
	\end{bmatrix}.
	$$
	Let $\mu=ht_{-}h^{-1}$, then $\mu$ is similar to $t_{-}$ and
	
	$$
	S^{-1}TS=D=\begin{bmatrix}
		t_{+}&0\\
		0&\mu
	\end{bmatrix}.
	$$
	Since $T^{\ast}JT=J$ and $S^{\ast}JS=H_{0}$, we obtain
	$$
	D^{\ast}H_{0}D=H_{0}.
	$$
	Expanding the left-hand side directly,
	$$
	D^{\ast}H_{0}D=\begin{bmatrix}
		0&\overline{t_{+}}\mu\\
		\overline{\mu}t_{+}&0
	\end{bmatrix}.
	$$
	Hence $\overline{t_{+}}\mu=1$, which means $\mu=(\overline{t_{+}})^{-1}=ht_{-}h^{-1}$. Let
	$$
	L=L^{-1}=\frac{1}{\sqrt{2}}\begin{bmatrix}
		1&1\\
		1&-1
	\end{bmatrix}
	$$ and
	$$
	X_{0}=LS^{-1}=\frac{1}{\sqrt{2}}\begin{bmatrix}
		\overline{h^{-1}}\overline{z_{-}}+\overline{z_{+}}&-\overline{h^{-1}}-1\\
		\overline{h^{-1}}\overline{z_{-}}-\overline{z_{+}} & 1-\overline{h^{-1}}
	\end{bmatrix}.
	$$
	Since $L^{\ast}JL=H_{0}$, we have
	$$
	X_{0}^{\ast}JX_{0}=(S^{-1})^{\ast}L^{\ast}JLS^{-1}=(S^{-1})^{\ast}H_{0}S^{-1}=J,
	$$
	this implies $X_{0}\in U(1,1;\mathbb{H})$. From Lemma \ref{1}, there exists a quaternion $q$ of modulus $1$ such that
	$$
	qt_{+}q^{-1}=re^{\boldsymbol{i}\theta},\ \ r=|t_{+}|>1.
	$$
	Then $q\mu q^{-1}=r^{-1}e^{\boldsymbol{i}\theta}$. Let
	$$
	Q=\begin{bmatrix}
		q&0\\
		0&q
	\end{bmatrix}\in U(1,1;\mathbb{H}),\ \ W=QX_{0}\in U(1,1;\mathbb{H}),\ \ t=\log r>0.
	$$
	Then
	$$
	\cosh t=\frac{r+r^{-1}}{2}, \ \ \sinh t=\frac{r-r^{-1}}{2},
	$$
	and finally
	$$
	WTW^{-1}=e^{\boldsymbol{i}\theta}\begin{bmatrix}
		\cosh t&\sinh t\\
		\sinh t&\cosh t
	\end{bmatrix}.
	$$
	It is easy to see that $WTW^{-1}$ is a complex matrix in $U(1,1;\mathbb{H})$, which completes the proof.
\end{proof}
\begin{remark}
	The matrix $T$ in Theorem \ref{4} is diagonalizable in $GL(2,\mathbb{H})$, since it has two right-linearly independent eigendirections. However, $T$ is not diagonalizable in $U(1,1;\mathbb{H})$, because if
	$$
	E=\begin{bmatrix}
		\alpha&0\\
		0&\delta
	\end{bmatrix}\in U(1,1;\mathbb{H}),
	$$
	then $E^{\ast}JE=J$ is forced to give $|\alpha|=|\delta|=1$, which contradicts the fact that $T$ is compound loxodromic.
\end{remark}
\begin{example}
	Let
	\[
	T=\frac14\begin{bmatrix}5\jj&3\jj\\3\jj&5\jj\end{bmatrix}.
	\]
	The defining relations of $U(1,1;\mathbb H)$ are satisfied. Furthermore,
	\[
	\Delta(T)=\left|\frac34\jj-\overline{\frac34\jj}\right|^2
	=\frac94>0,
	\]
	and hence the associated M\"obius transformation $g_T$ is compound loxodromic. Put
	\[
	q=\frac{1-\kk}{\sqrt2},\qquad
	W=\begin{bmatrix}q&0\\0&q\end{bmatrix}\in U(1,1;\mathbb H).
	\]
	Since $q\jj\bar q=\ii$, we obtain
	\[
	WTW^{-1}=\frac14\begin{bmatrix}5\ii&3\ii\\3\ii&5\ii\end{bmatrix}
	=e^{\ii\pi/2}
	\begin{bmatrix}\cosh(\log2)&\sinh(\log2)\\
		\sinh(\log2)&\cosh(\log2)\end{bmatrix}.
	\]
	Thus Theorem \ref{4} holds here with $t=\log2$ and $\theta=\pi/2$.
\end{example}

\section{Complex representatives of parabolic elements in $U(1,1;\mathbb{H})$}
In this section we prove that every matrix $T$ whose related quaternionic M\"{o}bius  transformation $\psi(T)$  is parabolic is similar to a complex matrix in $U(1,1;\mathbb{H})$.

\begin{theorem}\label{233}
	Let $g(z)=(az+b)(cz+d)^{-1}$ be a simple parabolic quaternionic M\"{o}bius  transformation, $T=\begin{bmatrix}
		a & b \\
		c & d
	\end{bmatrix}$, then there exists a matrix $Z \in U(1,1;\mathbb{H})$ such that $ZTZ^{-1}$ is  complex.
\end{theorem}
\begin{proof}
	In this case we have $b=\overline{c}\neq 0,\ d_{0}^{2}=1$. The rest of the proof is exactly the same as that of Theorem \ref{2}, here $Z$ is the same as $X$ in Theorem \ref{2}, just denoted differently and
	\begin{align*}
		ZTZ^{-1}&=\begin{bmatrix}
			d_{0}+\ii\sqrt{1-d_{0}^{2}+|c|^{2}} & |c| \\
			|c| & d_{0}-\ii\sqrt{1-d_{0}^{2}+|c|^{2}}
		\end{bmatrix}\\
		&=\begin{bmatrix}
			d_{0}+\ii|c| & |c| \\
			|c| & d_{0}-\ii|c|
		\end{bmatrix}.
	\end{align*}
\end{proof}
\begin{example}
	Let
	\[
	T=\begin{bmatrix}1-\ii&-\jj\\
		\jj&1-\ii\end{bmatrix}.
	\]
	Then $b=\bar c\ne0$ and $d_0^2=1$. Moreover,
	\[
	|1-\ii|^2-|\jj|^2=1,
	\qquad
	\overline{(1-\ii)}(-\jj)-\overline{\jj}(1-\ii)=0,
	\]
	so $T\in U(1,1;\mathbb H)$ and its corresponding M\"obius transformation $g_T$ is
	simple parabolic. For
	\[
	X=\begin{bmatrix}\jj&0\\0&1\end{bmatrix},
	\]
	we have
	\[
	XTX^{-1}=\begin{bmatrix}1+\ii&1\\1&1-\ii\end{bmatrix}.
	\]
	It can be seen that $XTX^{-1}$ is a complex matrix.
\end{example}
Below we prove that the compound parabolic quaternionic M\"{o}bius  transformation case can also be reduced to a complex matrix by similarity transformations in $U(1,1;\mathbb{H})$. Before that, we need to note a fact, namely that the quaternionic M\"{o}bius  transformation of  compound parabolic type necessarily admits a generalized right eigenvector. This conclusion does not hold for general matrices. For example,
$$
A=I_{2},\ \ \lambda=1.
$$
Although $\lambda=1$ is a right eigenvalue with multiplicity $2$, if we set
$$
A\boldsymbol{w}=\boldsymbol{w}+\boldsymbol{v},
$$
then $\boldsymbol{w}=\boldsymbol{w}+\boldsymbol{v}$. That is, $\boldsymbol{v}=\boldsymbol{0}$, which contradicts the fact that a right eigenvector is nonzero. Therefore, the fact that a right eigenvalue is of multiplicity $2$ does not by itself ensure the existence of a generalized right eigenvector of order two, it is additionally necessary that the matrix be non-diagonalizable.
\begin{lemma}
	Let $g(z)=(az+b)(cz+d)^{-1}$ be a compound parabolic quaternionic M\"{o}bius  transformation, $T=\begin{bmatrix}
		a & b \\
		c & d
	\end{bmatrix}$. Then the generalized right eigenvector of $T$ must exist.
\end{lemma}
\begin{proof}
	It follows from Lemma \ref{3} that there is only one similarity equivalence class of the right eigenvalues of $T$ and all the right eigenvalues are of modulus one. Choose a complex representative in the similarity class of right eigenvalues,
	$$
	\lambda\in\mathbb{C},\ \ |\lambda|=1.
	$$
	According to quaternionic Jordan theory, a right eigenvalue of multiplicity $2$ gives rise to two distinct cases
	$$
	\begin{bmatrix}
		\lambda&0\\
		0&\lambda
	\end{bmatrix}\ \text{or}\ \begin{bmatrix}
		\lambda&1\\
		0&\lambda
	\end{bmatrix}.
	$$
	Since $T$ is compound parabolic, it has a unique boundary fixed point. Consequently, it must have a Jordan block of size $2$
	$$
	P^{-1}TP=\begin{bmatrix}
		\lambda&1\\
		0&\lambda
	\end{bmatrix}.
	$$
	If we let $P=(\boldsymbol{v},\boldsymbol{w})$, then
	$$
	TP=P\begin{bmatrix}
		\lambda&1\\
		0&\lambda
	\end{bmatrix}.
	$$
	Compare the column vectors on both sides of the equation
	\begin{equation}\label{5}
		T\boldsymbol{v}=\boldsymbol{v}\lambda,\ \ T\boldsymbol{w}=\boldsymbol{v}+\boldsymbol{w}\lambda.
	\end{equation}
	Therefore, $\boldsymbol{w}$ is the desired generalized right eigenvector.
	
	More explicitly, we can solve for the vector $\boldsymbol{w}$ from $T$. Let $\boldsymbol{v}=(v_{1},v_{2})^{T}$ and $\boldsymbol{w}=(w_{1},w_{2})^{T}$. The second equality in (\ref{5}) is
	
	\begin{equation}\label{6}
		\begin{cases}
			aw_{1}+bw_{2}-w_{1}\lambda=v_{1}\\
			cw_{1}+dw_{2}-w_{2}\lambda=v_{2}
		\end{cases}
	\end{equation}
	
	These are two quaternion equations, which amount to eight real equations. Because quaternion multiplication is non-commutative. A convenient computational formulation is the following eight-dimensional real-linear system. For $q=q_{0}+q_{1}\boldsymbol{i}+q_{2}\boldsymbol{j}+q_{3}\boldsymbol{k}\in\mathbb{H}$, let $L(q)$ and $R(q)$ denote the real matrices representing left and right multiplication, respectively,
	$$
	L(q)=\begin{bmatrix}
		q_{0}&-q_{1}&-q_{2}&-q_{3}\\
		q_{1}&q_{0}&-q_{3}&q_{2}\\
		q_{2}&q_{3}&q_{0}&-q_{1}\\
		q_{3}&-q_{2}&q_{1}&q_{0}\
	\end{bmatrix},\ \ R(q)=\begin{bmatrix}
		q_{0}&-q_{1}&-q_{2}&-q_{3}\\
		q_{1}&q_{0}&q_{3}&-q_{2}\\
		q_{2}&-q_{3}&q_{0}&q_{1}\\
		q_{3}&q_{2}&-q_{1}&q_{0}\
	\end{bmatrix}.
	$$
	This means
	$$
	[q\boldsymbol{x}]_{\mathbb{R}}=L(q)[\boldsymbol{x}]_{\mathbb{R}},\ \ [\boldsymbol{x}q]_{\mathbb{R}}=R(q)[\boldsymbol{x}]_{\mathbb{R}}
	$$
	for any $[\boldsymbol{x}]_{\mathbb{R}}=(x_{0},x_{1},x_{2},x_{3})^{T}$. Thus, equation (\ref{6}) is equivalent to
	\begin{equation}\label{7}
		\begin{bmatrix}
			L(a)-R(\lambda)&L(b)\\
			L(c)&L(d)-R(\lambda)
		\end{bmatrix}\begin{bmatrix}
			[w_{1}]_{\mathbb{R}}\\
			[w_{2}]_{\mathbb{R}}
		\end{bmatrix}=\begin{bmatrix}
			[v_{1}]_{\mathbb{R}}\\
			[v_{2}]_{\mathbb{R}}
		\end{bmatrix}.
	\end{equation}
	When $T$ is compound parabolic, the system is consistent.
\end{proof}
Notice that the solution  of (\ref{7}) is generally non-unique. If $\boldsymbol{w}$ is a solution, then, under suitable commutation conditions, $\boldsymbol{w}+\boldsymbol{v}s$ may yield another representative of the same J-chain. This non-uniqueness is exactly the flexibility that will be used in the normalization below.

Let $z$ be the unique fixed point of $g(z)=(az+b)(cz+d)^{-1}$ on the boundary, which satisfies
$$
zcz+zd-az-b=0,\ \ |z|=1.
$$
Let $t=cz+d$. Then
$$
T\begin{bmatrix}
	z\\
	1
\end{bmatrix}=\begin{bmatrix}
	z\\
	1
\end{bmatrix}t.
$$
Let $q$ be a unit quaternion chosen so that $q^{-1}tq=\lambda\in\mathbb{C}$ and set $\boldsymbol{v}=(z,1)^{T}q$. Then we obtain
$$
T\boldsymbol{v}=\boldsymbol{v}\lambda\ \ \text{and}\ \ \boldsymbol{v}^{\ast}J\boldsymbol{v}=0.
$$
Let $h=\boldsymbol{v}^{\ast}J\boldsymbol{w}$, we aim to prove that $h\neq 0$. It can be seen that the orthogonal complement of a nonzero vector $\boldsymbol{v}$ is precisely the quaternionic line spanned by itself. That is,
$$
(\boldsymbol{v}\mathbb{H})^{\perp_{J}}=\boldsymbol{v}\mathbb{H}.
$$
If $h=0$, then $\boldsymbol{w}\in (\boldsymbol{v}\mathbb{H})^{\perp_{J}}=\boldsymbol{v}\mathbb{H}$, which would make $\boldsymbol{v}$ and $\boldsymbol{w}$ right linearly dependent. This contradicts the invertibility of $P=(\boldsymbol{v},\boldsymbol{w})$ in the J-chain. Hence $h\neq 0$.

Next we show that $h$ and $\lambda$ commute under multiplication. A direct computation yields
\begin{align*}
	h&=\langle\boldsymbol{w},\boldsymbol{v}\rangle_{J}\\
	&=\langle T\boldsymbol{w},T\boldsymbol{v}\rangle_{J}\\
	&=\langle\boldsymbol{w}\lambda+\boldsymbol{v},\boldsymbol{v}\lambda\rangle_{J}\\
	&=\overline{\lambda}h\lambda+\overline{\lambda}(\boldsymbol{v}^{\ast}J\boldsymbol{v})\\
	&=\overline{\lambda}h\lambda.
\end{align*}
Since $|\lambda|=1$, multiplying both sides of the equation on the left by $\lambda$ gives
\begin{equation}\label{8}
	\lambda h=h\lambda.
\end{equation}
Therefore, $h^{-1}$ also commutes multiplicatively with $\lambda$.

Define $\widehat{\boldsymbol{w}}=\boldsymbol{w}h^{-1}$, then
$$
\boldsymbol{v}^{\ast}J\widehat{\boldsymbol{w}}=hh^{-1}=1.
$$
Together with (\ref{8}), this gives
\begin{align*}
	T\widehat{\boldsymbol{w}}&=(T\boldsymbol{w})h^{-1}\\
	&=(\boldsymbol{w}\lambda+\boldsymbol{v})h^{-1}\\
	&=\boldsymbol{w}h^{-1}\lambda+\boldsymbol{v}h^{-1}\\
	&=\widehat{\boldsymbol{w}}\lambda+\boldsymbol{v}\beta,
\end{align*}
where $\beta=h^{-1}$.

Furthermore, since the Hermitian form takes real values on the same vector, we let $m=\widehat{\boldsymbol{w}}^{\ast}J\widehat{\boldsymbol{w}}\in\mathbb{R}$, 
$$
\boldsymbol{v}_{1}=\boldsymbol{v}\ \ \text{and} \ \ \boldsymbol{w}_{1}=\widehat{\boldsymbol{w}}-\frac{m}{2}\boldsymbol{v}.
$$
Then we have $\boldsymbol{v}_{1}^{\ast}J\boldsymbol{v}_{1}=0$ and
$$
\boldsymbol{v}_{1}^{\ast}J\boldsymbol{w}_{1}=\boldsymbol{v}^{\ast}J\widehat{\boldsymbol{w}}-\frac{m}{2}\boldsymbol{v}^{\ast}J\boldsymbol{v}=1.
$$
Moreover, because $\widehat{\boldsymbol{w}}^{\ast}J\boldsymbol{v}=\overline{\boldsymbol{v}^{\ast}J\widehat{\boldsymbol{w}}}=1$, we have
\begin{align*}
	\boldsymbol{w}_{1}^{\ast}J\boldsymbol{w}_{1}&=(\widehat{\boldsymbol{w}}-\frac{m}{2}\boldsymbol{v})^{\ast}J(\widehat{\boldsymbol{w}}-\frac{m}{2}\boldsymbol{v})\\
	&=m-\frac{m}{2}\widehat{\boldsymbol{w}}^{\ast}J\boldsymbol{v}-\frac{m}{2}\boldsymbol{v}^{\ast}J\widehat{\boldsymbol{w}}+\frac{m^{2}}{4}\boldsymbol{v}^{\ast}J\boldsymbol{v}\\
	&=m-\frac{m}{2}-\frac{m}{2}+0=0.
\end{align*}
Therefore
\begin{equation}\label{9}
	\boldsymbol{v}_{1}^{\ast}J\boldsymbol{v}_{1}=0,\ \ \boldsymbol{w}_{1}^{\ast}J\boldsymbol{w}_{1}=0,\ \ \boldsymbol{v}_{1}^{\ast}J\boldsymbol{w}_{1}=1.
\end{equation}
We have now completed the normalization. Notice that $\frac{m}{2}$ is real, it commutes multiplicatively with $\lambda$. Hence we obtain
\begin{equation}\label{10}
	T\boldsymbol{w}_{1}=\boldsymbol{w}_{1}\lambda+\boldsymbol{v}_{1}\beta,\ \ \beta=h^{-1}.
\end{equation}
In other words, subtracting $\frac{m}{2}\boldsymbol{v}$  affects only the J-norm of $\boldsymbol{w}_{1}$, leaving the chain coefficient $\beta$ unchanged. By virtue of (\ref{9}) and $T^{\ast}JT=J$, we compute the J-norm on both sides of (\ref{10}),
\begin{align*}
	0&=(T\boldsymbol{w}_{1})^{\ast}J(T\boldsymbol{w}_{1})\\
	&=(\boldsymbol{w}_{1}\lambda+\boldsymbol{v}_{1}\beta)^{\ast}J(\boldsymbol{w}_{1}\lambda+\boldsymbol{v}_{1}\beta)\\
	&=\overline{\beta}\lambda+\overline{\lambda}\beta.
\end{align*}

Therefore, $\gamma=\overline{\lambda}\beta$ must be a purely imaginary quaternion. Furthermore, since $\beta=h^{-1}$ commutes multiplicatively with $\lambda$, $\gamma$ commutes multiplicatively with $\lambda$. Put
$$
u=|\gamma|=|\beta|=|h|^{-1}>0,\ \ I=\frac{\gamma}{|\gamma|}.
$$
Then $I^2=-1$, $\gamma=uI$, and $\lambda\in\mathbb C_I$. Choose a
unit quaternion $\rho$ such that $\rho^{-1}I\rho=\ii$, and set
\[
\boldsymbol{v_2}=\boldsymbol{v_1}\rho,\quad \boldsymbol{w_2}=\boldsymbol{w_1}\rho,\quad
\lambda_2=\rho^{-1}\lambda\rho,\quad
\beta_2=\rho^{-1}\beta\rho.
\]
Then
\[
\overline{\lambda_2}\beta_2=\rho^{-1}(\overline{\lambda}\beta)\rho=u\boldsymbol{i},\qquad  T\boldsymbol{v_2}=\boldsymbol{v_2}\lambda_2,\qquad
T\boldsymbol{w_2}=\boldsymbol{w_2}\lambda_2+\boldsymbol{v_2}\beta_2,
\]
and
\[
\boldsymbol{v_2}^*J\boldsymbol{v_2}=\boldsymbol{w_2}^*J\boldsymbol{w_2}=0,\qquad \boldsymbol{v_2}^*J\boldsymbol{w_2}=1.
\]
Furthermore,
\[
\overline{\lambda_2}\beta_2
=\rho^{-1}(\bar\lambda\beta)\rho
=u\ii,
\]
and hence $\beta_2=\lambda_2u\ii$. Relabelling
$\boldsymbol{v_2},\boldsymbol{w_2},\lambda_2$ as $\boldsymbol{v_1},\boldsymbol{w_1},\lambda$, respectively, gives
\[
T\boldsymbol{w_1}=\boldsymbol{w_1}\lambda+\boldsymbol{v_1}(\lambda u\ii),
\qquad u>0.
\]

\begin{theorem}\label{229}
	Let $g(z)=(az+b)(cz+d)^{-1}$ be a compound parabolic quaternionic M\"{o}bius  transformation, $T=\begin{bmatrix}
		a & b \\
		c & d
	\end{bmatrix}$, then there exists a matrix $Y \in U(1,1;\mathbb{H})$ such that $YTY^{-1}$ is  complex.
\end{theorem}
\begin{proof}
	In this case, $b\neq\overline{c}$ and $\triangle(T)=\triangle=0$. Let $(\boldsymbol{v},\boldsymbol{w})$ be the $J$-normalized right $J$-chain constructed
	above, with chain coefficient $\lambda u\mathrm i$. Thus
	\[
	T\boldsymbol{v}=\boldsymbol{v}\lambda,\qquad
	T\boldsymbol{w}=\boldsymbol{w}\lambda+\boldsymbol{v}(\lambda u\ii),
	\]
	and
	\[
	\boldsymbol{v}^*J\boldsymbol{v}=\boldsymbol{w}^*J\boldsymbol{w}=0,\qquad \boldsymbol{v}^*J\boldsymbol{w}=1.
	\]
	
	Let
	$$
	P=\begin{bmatrix}
		v_{1}&w_{1}\\
		v_{2}&w_{2}
	\end{bmatrix}.
	$$
	Then
	$$
	P^{\ast}JP=H_{0},\ \ P^{-1}=H_{0}P^{\ast}J=\begin{bmatrix}
		\overline{w_{1}}&-\overline{w_{2}}\\
		\overline{v_{1}}&-\overline{v_{2}}
	\end{bmatrix}.
	$$
	Therefore,
	$$
	D_{p}\triangleq P^{-1}TP=\begin{bmatrix}
		\lambda&\lambda u\boldsymbol{i}\\
		0&\lambda
	\end{bmatrix}.
	$$
	Define a fixed matrix $C$ as
	$$
	C=\begin{bmatrix}
		1&\frac{1}{2}\\
		-\boldsymbol{i}&\frac{\boldsymbol{i}}{2}
	\end{bmatrix}.
	$$
	It can be verified that
	$$
	C^{-1}=\begin{bmatrix}
		\frac{1}{2}&\frac{\boldsymbol{i}}{2}\\
		1&-\boldsymbol{i}
	\end{bmatrix},\ \ C^{\ast}JC=H_{0}.
	$$
	Let $Y=CP^{-1}$, that is,
	$$
	Y=\begin{bmatrix}
		\overline{w_{1}}+\frac{1}{2}\overline{v_{1}}&-\overline{w_{2}}-\frac{1}{2}\overline{v_{2}}\\
		-\boldsymbol{i}\overline{w_{1}}+\frac{\boldsymbol{i}}{2}\overline{v_{1}}&\boldsymbol{i}\overline{w_{2}}-\frac{\boldsymbol{i}}{2}\overline{v_{2}}
	\end{bmatrix}.
	$$
	Then $Y$ is invertible and
	$$
	YTY^{-1}=CP^{-1}TPC^{-1}=CD_{p}C^{-1}=\lambda\begin{bmatrix}
		1+\boldsymbol{i}u&u\\
		u&1-\boldsymbol{i}u
	\end{bmatrix}\triangleq B.
	$$
	It is immediate that $B$ is complex. On the other hand,
	$$
	Y^{\ast}JY=(P^{-1})^{\ast}C^{\ast}JCP^{-1}=(P^{-1})^{\ast}H_{0}P^{-1}=J,
	$$
	this implies $Y\in U(1,1;\mathbb{H})$. This completes our proof.
\end{proof}

\begin{example}
	Let
	\[
	T=\begin{bmatrix}-1+\jj&\jj\\
		\jj&1+\jj\end{bmatrix}.
	\]
	One checks directly that $T\in U(1,1;\mathbb{H})$. Moreover, $b\ne\bar c$ and
	\[
	\Delta(T)=|\jj-(-\jj)|^2-((-1)-1)^2=0.
	\]
	Thus the corresponding M\"obius transformation $g_T$ is compound parabolic. Let
	\[
	q=\frac{1-\kk}{\sqrt2},\qquad
	W=\begin{bmatrix}q&0\\0&q\end{bmatrix}.
	\]
	Since $q\jj\bar q=\ii$, it follows that
	\[
	WTW^{-1}=\begin{bmatrix}-1+\ii&\ii\\
		\ii&1+\ii\end{bmatrix}
	=\ii\begin{bmatrix}1+\ii&1\\1&1-\ii\end{bmatrix}.
	\]
	This is the normal form in Theorem \ref{229} with $\lambda=\ii$ and $u=1$.
\end{example}

\section{Complete similarity invariants in $U(1,1;\mathbb{H})$}
In this section, we give complete similarity invariants for matrices in $U(1,1;\mathbb{H})$.
\begin{lemma}[\cite{YHF}]
	For $T\in U(1,1;\mathbb{H})$, $\triangle(T)$ and $Tr(T)$ are both similarity invariants.
\end{lemma}
\begin{lemma}\label{111}
	For $T=\begin{bmatrix}
		a&b\\
		c&d
	\end{bmatrix}\in U(1,1;\mathbb{H})$, if $b=\overline{c}=0$, then $T$ is similar to
	$$
	\begin{bmatrix}
		\xi&0\\
		0&\varphi
	\end{bmatrix}
	$$
	in $U(1,1;\mathbb{H})$, where $\xi$ and $\varphi$ are two representative elements of $\sigma_{R}(T)$. Moreover, if $a_0\neq d_0$, $\xi$ and $\varphi$ are not similar. If $a_0=d_0$, $\xi$ and $\varphi$ are similar.
\end{lemma}
The proof of Lemma \ref{111} is straightforward, so we omit it.
\begin{lemma}\label{21}
	Suppose that $a\in\mathbb{R},\ u>0,\ v>0$ such that $s=\sqrt{1-a^2+u^2}\geq 0,\ t=\sqrt{1-a^2+v^2}\geq 0$. Then the two matrices
	$$
	T_u=\begin{bmatrix}
		a+\boldsymbol{i}s&u\\
		u&a-\boldsymbol{i}s
	\end{bmatrix},\ \ T_v=\begin{bmatrix}
		a+\boldsymbol{i}t&v\\
		v&a-\boldsymbol{i}t
	\end{bmatrix}
	$$
	are both in $U(1,1;\mathbb{H})$ and they are similar in $U(1,1;\mathbb{H})$.
\end{lemma}
\begin{proof}
	First we prove $T_u$ and $T_v$ are in $U(1,1;\mathbb{H})$. Recall that the matrix
	$$
	J=\begin{bmatrix}
		1&0\\
		0&-1
	\end{bmatrix}.
	$$
	A simple calculation shows that
	$$
	T_{u}^{\ast}JT_{u}=J,\ \ T_{v}^{\ast}JT_{v}=J,
	$$
	which implies that $T_u$ and $T_v$ are in $U(1,1;\mathbb{H})$. Next, we aim to find a matrix $S$ in $U(1,1;\mathbb{H})$ such that $ST_{u}S^{-1}=T_{v}$.
	
	Let
	$$
	S=\frac{1}{2\sqrt{pq}}\begin{bmatrix}
		p+q&\boldsymbol{i}(q-p)\\
		-\boldsymbol{i}(q-p)&p+q
	\end{bmatrix},
	$$
	where
	$$
	p=u+s>0,\ \ q=v+t>0.
	$$
	It can be verified that $S^{\ast}JS=J$ and
	$$
	S^{-1}=\begin{bmatrix}
		\frac{p+q}{2\sqrt{pq}}&-\boldsymbol{i}\frac{q-p}{2\sqrt{pq}}\\
		\boldsymbol{i}\frac{q-p}{2\sqrt{pq}}&\frac{p+q}{2\sqrt{pq}}
	\end{bmatrix}.
	$$
	That is, $S\in U(1,1;\mathbb{H})$. Therefore,
	$$
	ST_{u}S^{-1}=\frac{1}{2\sqrt{pq}}\begin{bmatrix}
		p+q&\boldsymbol{i}(q-p)\\
		-\boldsymbol{i}(q-p)&p+q
	\end{bmatrix}\begin{bmatrix}
		a+\boldsymbol{i}s&u\\
		u&a-\boldsymbol{i}s
	\end{bmatrix}\begin{bmatrix}
		\frac{p+q}{2\sqrt{pq}}&-\boldsymbol{i}\frac{q-p}{2\sqrt{pq}}\\
		\boldsymbol{i}\frac{q-p}{2\sqrt{pq}}&\frac{p+q}{2\sqrt{pq}}
	\end{bmatrix}=T_{v},
	$$
	which shows that $T_u$ is similar to $T_v$ in $U(1,1;\mathbb{H})$.
\end{proof}
\begin{definition}
	Suppose that $T=\begin{bmatrix}
		a&b\\
		c&d
	\end{bmatrix}\in U(1,1;\mathbb{H})$. Let $r=a_0-d_0,\ q=b-\overline{c}$ and
	$$
	X_{T}=T+T^{-1}-\frac{Tr(T)}{2}I_{2},\ \ H_T=JX_T.
	$$
	Define
	$$
	\eta(T)=\begin{cases}
		\text{sgn}~r,\ \triangle(T)\leq 0,\\
		0,\ \triangle(T)>0,
	\end{cases}\ \ \kappa(T)=\text{rank}_{\mathbb{H}}\Big(T-\frac{Tr(T)}{4}I_2\Big).
	$$
\end{definition}
By definition, $H_T$ is a quaternionic Hermitian matrix. That is, the right eigenvalues of $H_T$ are all real numbers. Below we give the definition of the inertia index $\text{In}$ of a quaternionic Hermitian matrix $H$.
\begin{definition}
	Suppose that $H$ is a quaternionic Hermitian matrix. Define
	$$
	\text{In}(H)=(n_{+}(H),n_{-}(H),n_{0}(H)),
	$$
	where $n_{+}(H)$, $n_{-}(H)$ and $n_{0}(H)$ denote the numbers of positive, negative and zero right eigenvalues of $H$, respectively. Here, the right eigenvalues are counted with multiplicity.
\end{definition}
$\text{In}(H)$ records only the signs and the numbers of the right eigenvalues, not their specific magnitudes. As defined above, we may obtain
\begin{enumerate}
	\item
	$$
	X_T=\begin{bmatrix}
		r&q\\
		-\overline{q}&-r
	\end{bmatrix},\ \ X_{T}^2=-\triangle(T)I_2,\ \ H_T=\begin{bmatrix}
		r&q\\
		\overline{q}&r
	\end{bmatrix}=H_{T}^{\ast}.
	$$
	\item For any $S\in U(1,1;\mathbb{H})$, if we set $T^{\prime}=STS^{-1}$, then
	$$
	X_{T^{\prime}}=SX_{T}S^{-1},\ \ H_{T^{\prime}}=(S^{-1})^{\ast}H_{T}S^{-1}.
	$$
\end{enumerate}
\begin{definition}
	For a quaternionic Hermitian matrix $H$, denote
	$$
	\Phi_{H}(z)=z^{\ast}Hz\in\mathbb{R}.
	$$
	If $W\subseteq \mathbb{H}^{1,1}$ is a right linear subspace and
	$$
	z^{\ast}Hz>0,\ \ \forall\ 0\neq z\in W,
	$$
	then we say $W$ is $H$-positive definite. Similarly, if $V\subseteq \mathbb{H}^{1,1}$ is a right linear subspace and
	$$
	z^{\ast}Hz<0,\ \ \forall\ 0\neq z\in V,
	$$
	then we say $V$ is $H$-negative definite.
\end{definition}
\begin{lemma}
	Let $T\in U(1,1;\mathbb{H})$ and $H_T$ be defined as above. Then
	\begin{equation}\label{333}
		n_{+}(H_T)=\max\{ \dim_{\mathbb{H}}~W:W \text{is an} ~H_{T}\text{-positive definite subspace}\}.
	\end{equation}
	\begin{equation}\label{344}
		n_{-}(H_T)=\max\{ \dim_{\mathbb{H}}~V:V \text{is an} ~H_{T}\text{-negative definite subspace}\}.
	\end{equation}
\end{lemma}
\begin{proof}
	First, we compute the two right eigenvalues of $H_T$. If $q=0$, then $H_T=rI_2$. Its two right eigenvalues are both $r$. If $q\neq 0$, then its two right eigenvalues are $\lambda_{+}=r+|q|,\ \lambda_{-}=r-|q|$. Therefore, the right linear subspace spanned by the right eigenvectors corresponding to the positive right eigenvalues of $H_T$ is necessarily $H_{T}$-positive definite, and its dimension is exactly $n_{+}(H_T)$. Hence the right-hand side of equation (\ref{333}) is at least $n_{+}(H_T)$.
	
	On the other hand, the subspace $V_{\leq 0}$ spanned by the right eigenvectors corresponding to the non-positive right eigenvalues satisfies
	$$
	\dim_{\mathbb{H}}V_{\leq 0}=2-n_{+}(H_T),\ \ z^{\ast}H_{T}z\leq 0,\ \ z\in V_{\leq 0}.
	$$
	If there exists a $H_{T}$-positive definite subspace $W$ such that
	$$
	\dim_{\mathbb{H}}W>n_{+}(H_T),
	$$
	then
	$$
	\dim_{\mathbb{H}}(W\cap V_{\leq 0})\geq \dim_{\mathbb{H}}W+\dim_{\mathbb{H}}V_{\leq 0}-2>0.
	$$
	Hence there is a nonzero $z\in W\cap V_{\leq 0}$ satisfying
	$$
	z^{\ast}H_{T}z>0,\ \ z^{\ast}H_{T}z\leq 0,
	$$
	which is a contradiction. Therefore, (\ref{333}) holds. Using the same idea, we can prove (\ref{344}).
\end{proof}
Equations (\ref{333}) and (\ref{344}) apply equally to $H_{T^{\prime}}$.
\begin{theorem}\label{4.8}
	$\text{In}(H_T),\eta(T)$ and $\kappa(T)$ are all similarity invariants in $U(1,1;\mathbb{H})$. Once $\triangle(T)$ is fixed, $\text{In}(H_T)$ and $\eta(T)$ determine each other.
\end{theorem}
\begin{proof}
	First we prove that if $T^{\prime}=STS^{-1}$, then
	$$
	\text{In}(H_T)=\text{In}(H_{T^{\prime}}).
	$$
	For any $z\in\mathbb{H}^{1,1}$,
	$$
	z^{\ast}H_{T^{\prime}}z=(S^{-1}z)^{\ast}H_{T}(S^{-1}z).
	$$
	Therefore, $W$ is a $H_{T^{\prime}}$-positive definite subspace if and only if $S^{-1}W$ is a $H_{T}$-positive definite subspace. Since $S$ is invertible, this correspondence preserves the quaternionic dimension. By (\ref{333}),
	$$
	n_{+}(H_T)=n_{+}(H_{T^{\prime}}).
	$$
	Similarly, $n_{-}(H_T)=n_{-}(H_{T^{\prime}})$. For the zero inertia index, it follows from the invertibility of $(S^{-1})^{\ast}$,
	$$
	H_{T^{\prime}}z=0~\Leftrightarrow~H_{T}S^{-1}z=0.
	$$
	Therefore,
	$$
	\ker H_{T^{\prime}}=S(\ker H_{T}),
	$$
	consequently,
	$$
	n_{0}(H_{T^{\prime}})=\dim_{\mathbb{H}}\ker H_{T^{\prime}}=\dim_{\mathbb{H}}~\ker H_{T}=n_{0}(H_{T}).
	$$
	From the above, $\text{In}(H_T)=\text{In}(H_{T^{\prime}})$. Next we prove the similar invariance of $\eta(T)$. We further discuss the right eigenvalues of $H_T$ by cases.
	\begin{enumerate}
		\item [1.]$\triangle(T)<0$. In this case, $|q|<|r|$. Therefore, both right eigenvalues are nonzero and have the same sign as $r$.
		\item [2.]$\triangle(T)>0$. In this case, $|q|>|r|$. Therefore,
		$$
		\lambda_{+}=r+|q|>0,\ \ \lambda_{-}=r-|q|<0.
		$$
		\item [3.]$\triangle(T)=0$. In this case, $|q|=|r|$. The two right eigenvalues are $0$ and $2r$. Therefore, one of the right eigenvalues has the same sign as $r$.
	\end{enumerate}
	This yields the following table.

	\begin{table}[htbp]
		\centering
		\label{tab:inertia}
		\renewcommand{\arraystretch}{1.2}
		\begin{tabular}{|c|c|c|c|}
			\hline
			\multicolumn{2}{|c|}{}
			& $\operatorname{In}(H_T)$ & $\eta(T)$ \\
			\hline
			
			\multirow{2}{*}{$\triangle(T)<0$}
			& $r>0$ & $(2,0,0)$ & $1$ \\
			\cline{2-4}
			& $r<0$ & $(0,2,0)$ & $-1$ \\
			\hline
			
			$\triangle(T)>0$
			&  $\forall \ r$ & $(1,1,0)$ & $0$ \\
			\hline
			
			\multirow{3}{*}{$\triangle(T)=0$}
			& $r>0$ & $(1,0,1)$ & $1$ \\
			\cline{2-4}
			& $r<0$ & $(0,1,1)$ & $-1$ \\
			\cline{2-4}
			& $r=0$ & $(0,0,2)$ & $0$ \\
			\hline
		\end{tabular}
	\end{table}

	Therefore, under the condition that $\triangle(T)$ is fixed, $\text{In}(H_T)$ and $\eta(T)$ determine each other. This further shows that $\eta(T)$ is invariant under similarity in $U(1,1;\mathbb{H})$.
	
	Finally we prove that $\kappa(T)$ is a similarity invariant in $U(1,1;\mathbb{H})$. It follows from $T^{\prime}=STS^{-1}$ and $Tr(T^{\prime})=Tr(T)$ that
	$$
	T^{\prime}-\frac{Tr(T)}{4}I_2=S(T-\frac{Tr(T)}{4}I_2)S^{-1}.
	$$
	Since $S^{-1}$ maps $\mathbb{H}^{1,1}$ bijectively onto itself, the above equation gives the corresponding relation between the image spaces
	$$
	\text{Im}(T^{\prime}-\frac{Tr(T)}{4}I_2)=S\Big(\text{Im}(T-\frac{Tr(T)}{4}I_2)\Big).
	$$
	The invertible right linear map $S$ preserves the quaternionic dimension, hence
	\begin{align*}
		\kappa(T^{\prime})&=\dim_{\mathbb{H}}\text{Im}(T^{\prime}-\frac{Tr(T)}{4}I_2)\\
		&=\dim_{\mathbb{H}}\text{Im}(T-\frac{Tr(T)}{4}I_2)\\
		&=\kappa(T).
	\end{align*}
	That is, $\kappa(STS^{-1})=\kappa(T)$. This concludes the proof of all assertions of the theorem.
\end{proof}

For $T\in U(1,1;\mathbb{H})$, denote $(Tr(T),\triangle(T),\eta(T),\kappa(T))$ as $\mathcal{I}(T)$.
\begin{theorem}
	Suppose that $T_1=\begin{bmatrix}
		a&b\\
		c&d
	\end{bmatrix},T_2=\begin{bmatrix}
		e&f\\
		g&h
	\end{bmatrix}\in U(1,1;\mathbb{H})$. Then $T_1$ is similar to $T_2$ in $U(1,1;\mathbb{H})$ if and only if
	$$
	\mathcal{I}(T_1)=\mathcal{I}(T_2).
	$$
\end{theorem}
\begin{proof}
	The necessity follows from Theorem \ref{4.8}. We aim to prove the sufficiency.
	
	It follows from $\mathcal{I}(T_1)=\mathcal{I}(T_2)$ that
	$$
	Tr(T_1)=Tr(T_2),\ \triangle(T_1)=\triangle(T_2),\ \eta(T_1)=\eta(T_2),\ \kappa(T_1)=\kappa(T_2).
	$$
	We discuss by cases.
	
	\textbf{Case 1}: $\triangle(T_1)=\triangle(T_2)>0$. Then $g_{T_1}$ and $g_{T_2}$ are both compound loxodromic. Since $Tr(T_1)=Tr(T_2)$, by Lemma \ref{3},
	$$
	\sigma_{R}(T_1)=\sigma_{R}(T_2).
	$$
	Let
	$$
	t_{+}\in \sigma_{R}(T_1)=\sigma_{R}(T_2), \ \ |t_{+}|>1.
	$$
	Then there exists a nonzero quaternion $q$ such that
	$$
	qt_{+}q^{-1}=re^{\boldsymbol{i}\theta},\ \ r=|t_{+}|>1.
	$$
	Set
	$$
	\cosh t=\frac{r+r^{-1}}{2}, \ \ \sinh t=\frac{r-r^{-1}}{2}.
	$$
	Then, it follows from Theorem \ref{4} that $T_1$ and $T_2$ are similar to the same matrix $e^{\boldsymbol{i}\theta}\begin{bmatrix}
		\cosh t&\sinh t\\
		\sinh t&\cosh t
	\end{bmatrix}$  simultaneously in $U(1,1;\mathbb{H})$, which shows that they have the same complex representative in $U(1,1;\mathbb{H})$ under similarity, and therefore they are similar in $U(1,1;\mathbb{H})$.

	\textbf{Case 2}: $\triangle(T_1)=\triangle(T_2)<0$. Then $g_{T_1}$ and $g_{T_2}$ are both compound elliptic. 
	Set
	\[
	m=a_0+d_0=e_0+h_0
	=\frac{\operatorname{Tr}(T_1)}{2}
	=\frac{\operatorname{Tr}(T_2)}{2},
	\]
	and define
	\[
	x_{\pm}
	=\frac{1}{2}\left(m\pm\sqrt{-\triangle(T_1)}\right).
	\]
	Since $\triangle(T_1)<0$, we have $x_+>x_-$. By Lemma \ref{3},
	the right spectra of $T_1$ and $T_2$ consist of the same two
	similarity classes
	\[
	\sigma_R(T_1)=\sigma_R(T_2)
	=[\lambda_+]\cup[\lambda_-],
	\]
	where we may choose the complex representatives
	\[
	\lambda_{\pm}
	=x_{\pm}+\ii\sqrt{1-x_{\pm}^{2}}.
	\]
	Here $[\lambda]$ denotes the quaternionic similarity class of
	$\lambda$.
	
	Let
	\[
	r_1=a_0-d_0,\qquad r_2=e_0-h_0.
	\]
	Because
	\[
	\triangle(T_1)
	=|b-\bar c|^2-r_1^2<0,
	\qquad
	\triangle(T_2)
	=|f-\bar g|^2-r_2^2<0,
	\]
	both $r_1$ and $r_2$ are nonzero. Therefore,
	\[
	\eta(T_k)=\operatorname{sgn}(r_k),
	\qquad k=1,2.
	\]
	The equality $\eta(T_1)=\eta(T_2)$ implies only that
	\[
	\operatorname{sgn}(r_1)=\operatorname{sgn}(r_2),
	\]
	which is exactly the information needed to determine the order of
	the two spectral classes in the $U(1,1;\mathbb H)$-diagonal form.
	
	Indeed, by the ordered diagonalization obtained in Case 3 of
	Theorem 3.1 in \cite{YHF} (see also Remark 3.3 therein), for each
	$k=1,2$ there exists $S_k\in U(1,1;\mathbb H)$ such that
	\[
	S_kT_kS_k^{-1}
	=
	\begin{cases}
		\operatorname{diag}(\lambda_+,\lambda_-),
		& r_k>0,\\[2mm]
		\operatorname{diag}(\lambda_-,\lambda_+),
		& r_k<0.
	\end{cases}
	\]
	Since $r_1$ and $r_2$ have the same sign, the ordered diagonal
	matrices obtained from $T_1$ and $T_2$ are identical. Denote this
	common diagonal matrix by $D$. Then
	\[
	S_1T_1S_1^{-1}=D=S_2T_2S_2^{-1}.
	\]
	Consequently,
	\[
	T_2=(S_2^{-1}S_1)T_1(S_2^{-1}S_1)^{-1}.
	\]
	Since $S_2^{-1}S_1\in U(1,1;\mathbb H)$, it follows that $T_1$
	and $T_2$ are similar in $U(1,1;\mathbb H)$.
	
	\textbf{Case 3}: $\triangle(T_1)=\triangle(T_2)=0$. Notice that $\eta(T_1)=\eta(T_2)$.
	\begin{enumerate}
		\item If $\eta(T_1)=\eta(T_2)\neq 0$, then $g_{T_1}$ and $g_{T_2}$ are both compound parabolic. Since $Tr(T_1)=Tr(T_2)$, it follows from Lemma \ref{3} that
		$$
		\sigma_{R}(T_1)=\sigma_{R}(T_2).
		$$
		We choose a complex representative element $\lambda\in\sigma_{R}(T_1)=\sigma_{R}(T_2)$. Then by Theorem \ref{229}, $T_1$ and $T_2$ are similar to matrices
		$$
		T_u=\lambda\begin{bmatrix}
			1+\boldsymbol{i}u&u\\
			u&1-\boldsymbol{i}u
		\end{bmatrix},\ \  T_v=\lambda\begin{bmatrix}
			1+\boldsymbol{i}v&v\\
			v&1-\boldsymbol{i}v
		\end{bmatrix},\ \ u,v>0
		$$
		respectively in $U(1,1;\mathbb{H})$. By the conclusion of Lemma \ref{21}, $T_u$ is similar to $T_v$ in $U(1,1;\mathbb{H})$, and  it follows that $T_1$ and $T_2$ are similar in $U(1,1;\mathbb{H})$.
		
		\item If $\eta(T_1)=\eta(T_2)=0$, we next consider the value of $\kappa(T_1)=\kappa(T_2)$.
		\begin{enumerate}
			\item [(a)]If $\kappa(T_1)=\kappa(T_2)=0$, then $T_1=T_2=\pm I_2$. Thus they are trivially similar in $U(1,1;\mathbb{H})$.
			\item [(b)]If $\kappa(T_1)=\kappa(T_2)=1$, then neither $T_1$ nor $T_2$ is $\pm I_2$ and $Tr(T_1)=Tr(T_2)=\pm 4$. That is,
			$$
			a_0=d_0=e_0=h_0=\pm 1.
			$$
			In other words, $g_{T_1}$ and $g_{T_2}$ are both simple parabolic. Furthermore, Theorem \ref{233} shows that $T_1$ and $T_2$ are similar to matrices
			$$
			T_c=\begin{bmatrix}
				d_{0}+\ii|c| & |c| \\
				|c| & d_{0}-\ii|c|
			\end{bmatrix},\ \ T_g=\begin{bmatrix}
				d_{0}+\ii|g| & |g| \\
				|g| & d_{0}-\ii|g|
			\end{bmatrix},\ \ |c|,|g|>0
			$$
			respectively in $U(1,1;\mathbb{H})$. By the conclusion of Lemma \ref{21}, $T_c$ is similar to $T_g$ in $U(1,1;\mathbb{H})$, and hence $T_1$ and $T_2$ are similar in $U(1,1;\mathbb{H})$.
			\item [(c)] If $\kappa(T_1)=\kappa(T_2)=2$, then neither $T_1$ nor $T_2$ is $\pm I_2$. Now we focus on the range of $Tr(T_1)=Tr(T_2)$.
			\begin{enumerate}
				\item [(i)]If $|Tr(T_1)|=|Tr(T_2)|>4$, then $g_{T_1}$ and $g_{T_2}$ are both simple loxodromic. At this point, we have
				$$
				a_0=d_0=e_0=h_0,\ \ |a_0|>1.
				$$
				
				Theorem \ref{2} shows that $T_1$ and $T_2$ are similar to matrices
				$$
				T_b=\begin{bmatrix}
					d_{0}+\ii\sqrt{1-d_{0}^{2}+|c|^{2}} & |c| \\
					|c| & d_{0}-\ii\sqrt{1-d_{0}^{2}+|c|^{2}}
				\end{bmatrix}
				$$
				and
				$$
				T_f=\begin{bmatrix}
					h_{0}+\ii\sqrt{1-h_{0}^{2}+|g|^{2}} & |g| \\
					|g| & h_{0}-\ii\sqrt{1-h_{0}^{2}+|g|^{2}}
				\end{bmatrix}
				$$
				respectively in $U(1,1;\mathbb{H})$. By the conclusion of Lemma \ref{21}, $T_b$ is similar to $T_f$ in $U(1,1;\mathbb{H})$, and hence obtained that $T_1$ and $T_2$ are similar in $U(1,1;\mathbb{H})$.
				
				\item [(ii)] If $|Tr(T_1)|=|Tr(T_2)|<4$, then $g_{T_1}$ and $g_{T_2}$ are both simple elliptic. At this point, we have
				$$
				a_0=d_0=e_0=h_0,\ \ |a_0|<1.
				$$
				It follows from Lemma \ref{3} that
				$$
				\sigma_{R}(T_1)=\sigma_{R}(T_2).
				$$

				Lemma \ref{111} shows that $T_1$ and $T_2$ are similar to matrices
				$$
				T_{\lambda}=\begin{bmatrix}
					\lambda& 0\\
					0 & \lambda
				\end{bmatrix}
				$$
				and
				$$
				T_{\alpha}=\begin{bmatrix}
					\alpha & 0 \\
					0 & \alpha
				\end{bmatrix}
				$$
				respectively in $U(1,1;\mathbb{H})$. Since the common right spectrum $\sigma_{R}(T_1)=\sigma_{R}(T_2)$ consists of a single similarity
				class in this case, we have $\lambda\sim\alpha$. Then it is easy to verify that $T_{\lambda}$ is similar to $T_{\alpha}$ in $U(1,1;\mathbb{H})$, and finally $T_1$ and $T_2$ are similar in $U(1,1;\mathbb{H})$.
			\end{enumerate}
		\end{enumerate}
	\end{enumerate}
\end{proof}
\section{Recurrence formula of $\triangle(T^n)$ and classification of the iterates $g_{T^n}$}
In this section, we obtain the recurrence formula of $\triangle(T^n)$ by using Chebyshev polynomials. Furthermore, we provide a complete classification of the M\"{o}bius transformation $g_{T^n}$ by combining $\triangle(T^n)$ with several other quantities.
\subsection{Recurrence formula of $\triangle(T^n)$}
\begin{lemma}[\cite{YHF}]\label{88}
	For $T\in U(1,1;\mathbb{H})$, the following identity is satisfied
	\begin{equation}\label{888}
		\triangle(T)=\frac{1}{4}(Tr (T))^{2}-\frac{1}{2}Tr(T^{2})-2.
	\end{equation}
\end{lemma}

\begin{lemma}[Cayley-Hamilton Theorem \cite{book}]\label{25}
	For a polynomial $p(x)=\alpha_{0}+\alpha_{1}x+\cdots+\alpha_{k}x^{k}$ with complex coefficients, and a complex square matrix $A$, the corresponding matrix polynomial is defined to be
	$$
	p(A)=\alpha_{0}I+\alpha_{1}A+\cdots+\alpha_{k}A^{k}.
	$$
	Moreover, if $p(x)$ is the characteristic polynomial for $A$, then $p(A)=0$.
\end{lemma}
The Chebyshev polynomials of the first kind are
characterized by $C_{n}(\cos\theta)=\cos(n\theta)$. Or equivalently,
$$
C_{0}(t)=1,\ \ C_{1}(t)=t,\ \ C_{n+1}(t)=2tC_{n}(t)-C_{n-1}(t).
$$

For $A\in M_{n}(\mathbb{H})$, let $A$ be represented as $A=A_1+A_{2}\boldsymbol{j}$, where $A_1,A_2\in M_{n}(\mathbb{C})$. The complex representation of $A$ is defined to be
$$
\chi(A)=\begin{bmatrix}
	A_1&A_2\\
	-\overline{A_2}&\overline{A_1}
\end{bmatrix}.
$$
Then $\det_{\mathbb{C}}(\chi(A))=(\det_{\mathbb{H}}(A))^{2}$. The properties of $\chi(A)$ and $\det_{\mathbb{C}}(\chi(A))$ can be seen in \cite{FZZ1997}.

\begin{lemma}[\cite{FZZ1997}]\label{38}
	Let $A$ and $B$ be any $n\times n$ complex matrices and $C=\begin{bmatrix}
		A&B\\
		-\overline{B}&\overline{A}
	\end{bmatrix}$. Then $\det_{\mathbb{C}}(C)\geq 0$. Furthermore, $C$ is invertible~($\det_{\mathbb{C}}(C)>0$) in $M_{2n}(\mathbb{C})$ if and only if $A+B\boldsymbol{j}$ is invertible in $M_{n}(\mathbb{H})$.
\end{lemma}
For convenience, we set
$$
\tau=Tr(T),\ \ \sigma=Tr(T^{2}),\ \ \delta=\triangle(T).
$$
Then $\sigma=\frac{\tau^{2}}{2}-2\delta-4$ by (\ref{888}).

\begin{lemma}\label{29}
	Let $T\in U(1,1;\mathbb{H})$. Define $p_{T}(z)=\det(zI_{4}-\chi(T))$. Then
	$$
	p_{T}(z)=z^{4}-\tau z^{3}+\beta z^{2}-\tau z+1,\ \ \beta=\frac{\tau^{2}-\sigma}{2}=\frac{1}{4}\tau^{2}+\delta+2.
	$$
\end{lemma}
\begin{proof}
	Let $T=A+B\boldsymbol{j}$, where $A,B\in M_{2}(\mathbb{C})$ and
	$$
	M=\chi(T)=\begin{bmatrix}
		A&B\\
		-\overline{B}&\overline{A}
	\end{bmatrix}.
	$$
	
	The properties of complex representations guarantee that
	$$
	tr_{\mathbb{C}}(M)=\tau,\ \ tr_{\mathbb{C}}(M^{2})=\sigma.
	$$
	Let the four eigenvalues of the matrix $M\in M_{4}(\mathbb{C})$, counted with multiplicity, be denoted by $\lambda_{1},\lambda_{2},\lambda_{3},\lambda_{4}$. Expanding
	$$
	p_{T}(z)=\prod_{j=1}^{4}(z-\lambda_{j})=z^{4}-e_{1}z^{3}+e_{2}z^{2}-e_{3}z+e_{4}.
	$$
	Here $e_1$ denotes the sum of the four eigenvalues, $e_2$ the sum of their pairwise products, $e_3$ the sum of their triple products, and $e_4$ the product of all four eigenvalues. First of all,
	$$
	e_1=\lambda_{1}+\lambda_{2}+\lambda_{3}+\lambda_{4}=tr_{\mathbb{C}}(M)=\tau.
	$$
	Secondly, by expanding the square, we obtain
	$$
	(\sum_{j=1}^{4}\lambda_{j})^{2}=\sum_{j=1}^{4}\lambda_{j}^{2}+2\sum_{i<j}\lambda_{i}\lambda_{j}.
	$$
	Therefore,
	$$
	e_2=\frac{\tau^{2}-\sigma}{2}=\frac{\tau^{2}}{4}+\delta+2=\beta.
	$$
	It follows from the definition of $\det_{\mathbb{H}}(T)$ and Lemma \ref{38} that $\det(M)=1$.

	Since all eigenvalues are nonzero,
	\begin{align*}
		e_3&=\lambda_{1}\lambda_{2}\lambda_{3}+\lambda_{1}\lambda_{2}\lambda_{4}+\lambda_{1}\lambda_{3}\lambda_{4}+\lambda_{2}\lambda_{3}\lambda_{4}\\
		&=(\lambda_{1}\lambda_{2}\lambda_{3}\lambda_{4})(\lambda_{1}^{-1}+\lambda_{2}^{-1}+\lambda_{3}^{-1}+\lambda_{4}^{-1})\\
		&=\det(M)tr_{\mathbb{C}}(M^{-1}).
	\end{align*}
	
	Moreover, since
	$$
	T^{-1}=\begin{bmatrix}
		\overline{a}&-\overline{c}\\
		-\overline{b}&\overline{d}
	\end{bmatrix},
	$$
	we have
	$$
	tr_{\mathbb{C}}(M^{-1})=Tr(T^{-1})=\tau.
	$$
	Hence $e_3=1\cdot\tau=\tau$.
\end{proof}
If we decompose $p_{T}(z)$ over $\mathbb{C}$ into
$$
p_{T}(z)=(z^{2}-2xz+1)(z^{2}-2yz+1),
$$
then
\begin{equation}\label{26}
	x+y=\frac{\tau}{2},\ \ 4xy+2=\beta,\ \ (x-y)^{2}=-\delta.
\end{equation}

By Lemma \ref{25}, Lemma \ref{29} and the injectivity of $\chi$, we may obtain
$$
T^{4}-\tau T^{3}+\beta T^{2}-\tau T+I_{2}=0_{2},
$$
where $0_2$ denotes the zero matrix of order 2.
Consequently,
\begin{equation}\label{30}
	T^{n+4}=\tau T^{n+3}-\beta T^{n+2}+\tau T^{n+1}-T^{n},\ \ n\geq 0.
\end{equation}
Identity (\ref{30}) is the basis for the recurrence derived below.

\begin{lemma}
	$$
	Tr(T^{n})=2(C_{n}(x)+C_{n}(y)).
	$$
\end{lemma}
\begin{proof}
	Assume $\lambda_{1}$ and $\lambda_{2}$ are two eigenvalues of $z^{2}-2xz+1=0$. Then $\lambda_{1}+\lambda_{2}=2x,\ \ \lambda_{1}\lambda_{2}=1$. Hence both $\lambda_{1}$ and $\lambda_{2}$ are nonzero and $\lambda_{1}=\frac{1}{\lambda_{2}}$. We may rewrite them as $\lambda,\lambda^{-1}$. Similarly, the two eigenvalues of the other factor $z^{2}-2yz+1=0$ can be denoted as $\mu,\mu^{-1}$, then we have $\mu+\mu^{-1}=2y$. Therefore, the four eigenvalues of $p_{T}(z)$ are
	$$
	\lambda,\lambda^{-1},\ \mu,\ \mu^{-1}.
	$$
	Here, eigenvalues are counted with multiplicity, and some roots are allowed to be equal. In fact, these four eigenvalues are precisely the eigenvalues of the complex matrix $M$. Therefore, the four eigenvalues of $M^{n}$ are $\lambda^{n},\lambda^{-n},\mu^{n},\mu^{-n}$. Since $Tr(T^{n})=tr_{\mathbb{C}}(M^{n})$, we have
	\begin{equation}\label{22}
		Tr(T^{n})=\lambda^{n}+\lambda^{-n}+\mu^{n}+\mu^{-n}.
	\end{equation}
	Let $A_{n}=\lambda^{n}+\lambda^{-n}$. Its first two terms are
	$$
	A_{0}=2,\ \ A_{1}=\lambda+\lambda^{-1}=2x.
	$$
	Moreover, expanding the multiplication directly gives
	\begin{align*}
		A_{n+1}&=(\lambda+\lambda^{-1})(\lambda^{n}+\lambda^{-n})-(\lambda^{n-1}+\lambda^{-(n-1)})\\
		&=2xA_{n}-A_{n-1}.
	\end{align*}
	On the other hand, the first two terms of the sequence $2C_{n}(x)$ are also $2,2x$, and it satisfies the same recurrence relation. Therefore
	$$
	A_{n}=2C_{n}(x).
	$$
	Likewise, $\mu^{n}+\mu^{-n}=2C_{n}(y)$. Substitution into (\ref{22}) yields
	\begin{equation}\label{23}
		Tr(T^{n})=2(C_{n}(x)+C_{n}(y)).
	\end{equation}
\end{proof}

\begin{lemma}\label{27}
	$$
	\triangle(T^{n})=-(C_{n}(x)-C_{n}(y))^{2}.
	$$
\end{lemma}
\begin{proof}
	Every $T\in U(1,1;\mathbb{H})$ satisfies (\ref{888}) in Lemma \ref{88}. Since $T^{n}$ still belongs to this group, it follows that
	$$
	\triangle(T^{n})=\frac{1}{4}(Tr (T^{n}))^{2}-\frac{1}{2}Tr(T^{2n})-2.
	$$
	Replacing $n$ by $2n$ in (\ref{23}), we obtain
	$$
	Tr(T^{2n})=2(C_{2n}(x)+C_{2n}(y)).
	$$
	We now employ the double-angle identity for Chebyshev polynomials
	$$
	C_{2n}(t)=2C_{n}(t)^{2}-1.
	$$
	Therefore
	$$
	Tr(T^{2n})=4C_{n}(x)^{2}+4C_{n}(y)^{2}-4.
	$$

	And hence
	\begin{align*}
		\triangle({T^{n}})&=\frac{1}{4}[2(C_{n}(x)+C_{n}(y))]^{2}-\frac{1}{2}(4C_{n}(x)^{2}+4C_{n}(y)^{2}-4)-2\\
		&=(C_{n}(x)+C_{n}(y))^{2}-2C_{n}(x)^{2}-2C_{n}(y)^{2}+2-2\\
		&=-(C_{n}(x)-C_{n}(y))^{2}.
	\end{align*}	
\end{proof}

For convenience, we introduce the following definition of $q_n$. Assume $C_{n}(t)=\sum_{k=0}^{n}\gamma_{k}t^{k}$. By means of
$$
x^{k}-y^{k}=(x-y)\sum_{j=0}^{k-1}x^{k-1-j}y^{j}\ \ (k\geq 1),
$$
we obtain the bivariate polynomial corresponding to the divided difference
$$
q_{n}=\frac{C_{n}(x)-C_{n}(y)}{x-y}=\sum_{k=1}^{n}\gamma_{k}\sum_{j=0}^{k-1}x^{k-1-j}y^{j}.
$$
We adopt the convention that the empty sum equals $0$ when $n=0$. It is still well defined when $x=y$ and
$$
q_{n}=\sum_{k=1}^{n}k\gamma_{k}t^{k-1}=C_{n}^{\prime}(t).
$$
This allows us to give the complete definition of $q_n$.
\begin{definition}
	Define
	\begin{displaymath}
		q_n=\left \{ \begin{array}{ll}
			\frac{C_{n}(x)-C_{n}(y)}{x-y},\ x\neq y,\\
			C_{n}^{\prime}(x),\ x=y.
		\end{array} \right.
	\end{displaymath}
\end{definition}

For a sequence $a=(a_0,a_1,\cdots)$, denote
$$
Ea=(a_1,a_2,\cdots).
$$
That is, $(Ea)_n=a_{n+1},\ (E^{2}a)_{n}=a_{n+2},\cdots,(E^{k}a)_{n}=a_{n+k}$. The operator $E$ is called the shift operator.

\begin{lemma}\label{28}
	$q_n$ is a real number and uniquely determined by the initial values and the fourth-order recurrence:
	$$
	q_0=0,\ \ q_1=1,\ \ q_2=\tau,\ \ q_3=\frac{3}{4}\tau^{2}-\delta-3,
	$$
	\begin{equation}\label{36}
		q_{n+4}=\tau q_{n+3}-(\frac{1}{4}\tau^{2}+\delta+2)q_{n+2}+\tau q_{n+1}-q_{n}.
	\end{equation}
\end{lemma}
\begin{proof}
	First of all, we prove $q_n$ is real. (\ref{26}) tells us $x$ and $y$ are either both real or complex conjugates of each other. Since $C_n$ has real coefficients and $x,y$ are either both real or complex conjugates, the definition of $q_n$ shows that $q_n\in\mathbb{R}$ for every $n$. Next we present the derivation of the recurrence formula for $q_n$. We discuss the two cases $x\neq y$ and $x=y$ separately.
	
	\textbf{Case 1}: $x\neq y$. \ By the defining recurrence for the Chebyshev polynomials,
	$$
	C_{n+2}(x)-2xC_{n+1}(x)+C_{n}(x)=0,\ \ C_{n+2}(y)-2yC_{n+1}(y)+C_{n}(y)=0.
	$$
	In terms of the shift operator, these two recurrences become
	$$
	(E^{2}-2xE+I)a_n=0,\ \ (E^{2}-2yE+I)b_n=0,
	$$
	if we set $a_n=C_{n}(x),\ b_n=C_{n}(y)$. For brevity, let $L_x=E^{2}-2xE+I,\ L_y=E^{2}-2yE+I$ and $d_n=a_n-b_n$. As $L_x$ and $L_y$ are both polynomials with constant coefficients in the same operator $E$, they commute. That is, $L_{x}L_{y}=L_{y}L_{x}$. Therefore,
	$$
	L_{x}L_{y}a_{n}=L_{y}L_{x}a_{n}=L_{y}(0)=0,\ \ L_{x}L_{y}b_{n}=L_{y}L_{x}b_{n}=L_{x}(0)=0.
	$$
	By virtue of linearity,
	$$
	L_{x}L_{y}(a_{n}-b_{n})=L_{x}L_{y}d_{n}=0-0=0.
	$$
	Recall that
	\begin{align*}
		p_{T}(z)&=(z^{2}-2xz+1)(z^{2}-2yz+1)\\
		&=z^{4}-2(x+y)z^{3}+(4xy+2)z^{2}-2(x+y)z+1
	\end{align*}
	Hence
	$$
	d_{n+4}-2(x+y)d_{n+3}+(4xy+2)d_{n+2}-2(x+y)d_{n+1}+d_{n}=0.
	$$
	By (\ref{26}),
	$$
	d_{n+4}=\tau d_{n+3}-\beta d_{n+2}+\tau d_{n+1}-d_{n}.
	$$
	When $x\neq y$, dividing by the constant $x-y$, which does not depend on $n$, yields the  fourth-order recurrence in (\ref{36}) for $q_{n}$. Substituting $n=0,1,2,3$ respectively into the definition for $q_n$ when $x\neq y$, we obtain the four initial values $q_0=0,\ \ q_1=1,\ \ q_2=\tau,\ \ q_3=\frac{3}{4}\tau^{2}-\delta-3$.
	
	\textbf{Case 2}: $x=y$. \ It follows from (\ref{26}) that $t_0=x=y=\frac{\tau}{4}$ and $\delta=\triangle(T)=0$. Set $D_{n}(t)=C_{n}^{\prime}(t)$, then $q_{n}=D_{n}(t_{0})$. By definition of Chebyshev polynomials,
	$$
	C_{n+2}(t)-2tC_{n+1}(t)+C_{n}(t)=0,\ \ n\geq 0.
	$$
	Differentiating the above formula with respect to $t$ yields
	$$
	D_{n+2}(t)-2tD_{n+1}(t)+D_{n}(t)=2C_{n+1}(t).
	$$
	Evaluating at $t=t_0$, we obtain
	$$
	q_{n+2}-2t_{0}q_{n+1}+q_{n}=2C_{n+1}(t_{0}).
	$$
	Let $A_n=q_{n+2}-2t_{0}q_{n+1}+q_{n}$. Then $A_n=2C_{n+1}(t_{0})$. Since
	$$
	C_{n+3}(t_{0})-2t_{0}C_{n+2}(t_{0})+C_{n+1}(t_{0})=0,
	$$
	multiplying both sides by $2$ gives
	$$
	A_{n+2}-2t_{0}A_{n+1}+A_{n}=0.
	$$
	Substituting the definition of $A_n$ back in and expanding term by term, we obtain
	\begin{align*}
		0&=(q_{n+4}-2t_{0}q_{n+3}+q_{n+2})-2t_{0}(q_{n+3}-2t_{0}q_{n+2}+q_{n+1})+(q_{n+2}-2t_{0}q_{n+1}+q_{n})\\
		&=q_{n+4}-4t_{0}q_{n+3}+(4t_{0}^{2}+2)q_{n+2}-4t_{0}q_{n+1}+q_{n}.
	\end{align*}
	Therefore,
	\begin{align*}
		q_{n+4}&=4t_{0}q_{n+3}-(4t_{0}^{2}+2)q_{n+2}+4t_{0}q_{n+1}-q_{n}\\
		&=\tau q_{n+3}-(\frac{\tau^{2}}{4}+2)q_{n+2}+\tau q_{n+1}-q_{n}.
	\end{align*}
	This is precisely the formula in the general recurrence of $q_{n}$ when $\delta=0$. In this case, $\beta=\frac{\tau^{2}}{4}+\delta+2=\frac{\tau^{2}}{4}+2$. The four initial values of $q_n$ are then derived directly from the derivative of the polynomial $C_{n}(t)$, which are exactly the same as the four initial values in case $1$.
\end{proof}

\begin{theorem}\label{34}
	For each integer $n\geq 1$, the following equation always holds true.
	$$
	\triangle(T^{n})=q_{n}^{2}\triangle(T).
	$$
\end{theorem}
\begin{proof}
	This conclusion can be proved by Lemma \ref{27} and Lemma \ref{28}.
\end{proof}
\begin{corollary}
	$\triangle(T^{n})\triangle(T)\geq 0$.
\end{corollary}
Theorem \ref{34} gives an exact expression for $\triangle(T^{n})$, but it is insufficient to determine the type of $T^n$ when $\triangle(T^{n})$ is $0$. Next, we will provide a more detailed classification of $g_{T^n}$.
\subsection{Classification of $g_{T^n}$}

Write $T^{n}=\begin{bmatrix}
	a_{n}&b_{n}\\
	c_{n}&d_{n}
\end{bmatrix}$. Define
$$
u_{n}=b_{n}-\overline{c_{n}},\ \ v_{n}=(a_{n})_{0}-(d_{n})_{0}=\operatorname{Re}(a_n)
-\operatorname{Re}(d_n).
$$
The following proposition further remedies the deficiency of Theorem \ref{34}.
\begin{proposition}\label{35}
	$u_n$ and $v_n$ satisfy the same fourth-order recurrence as $q_n$, specifically,
	$$
	u_n=q_{n}(b-\overline{c})=q_{n}u_1,\ \ v_{n}=q_{n}(a_{0}-d_{0})=q_{n}v_1,\ \ \text{for any} \ n\geq 0.
	$$
\end{proposition}
\begin{proof}
	From the matrix recurrence relation (\ref{30}), comparing the entries of the matrices componentwise, we obtain
	\begin{equation}\label{31}
		b_{n+4}=\tau b_{n+3}-\beta b_{n+2}+\tau b_{n+1}-b_{n},
	\end{equation}
	\begin{equation}\label{32}
		c_{n+4}=\tau c_{n+3}-\beta c_{n+2}+\tau c_{n+1}-c_{n}.
	\end{equation}
	Since $\tau$ and $\beta$ are both real numbers, taking the conjugate of (\ref{32}) gives
	\begin{equation}\label{33}
		\overline{c_{n+4}}=\tau \overline{c_{n+3}}-\beta \overline{c_{n+2}}+\tau \overline{c_{n+1}}-\overline{c_{n}}.
	\end{equation}
	Subtracting (\ref{33}) from (\ref{31}) gives
	$$
	u_{n+4}=\tau u_{n+3}-\beta u_{n+2}+\tau u_{n+1}-u_{n}.
	$$
	Similarly, taking the real parts of the recurrence relations for $a_n$ and $d_n$, and then subtracting them, we obtain
	$$
	v_{n+4}=\tau v_{n+3}-\beta v_{n+2}+\tau v_{n+1}-v_{n}.
	$$
	Therefore, $u_n$ and $v_n$ indeed satisfy the same fourth-order recurrence as $q_n$; this fact follows from the matrix identity. Furthermore, direct calculation gives the first four terms of $u_n$, $v_n$ and $q_n$ as follows.
	
	\begin{table}[htbp]
		\centering
		\label{tab: quv}
		\renewcommand{\arraystretch}{1.2}
		\begin{tabular}{|c|c|c|c|}
			\hline
			n& $q_n$ & $u_n$ & $v_n$\\
			\hline
			0&0 & 0 &0\\
			
			1& 1 & $u_1$ &$v_1$\\
			
			2& $\tau$ & $\tau u_1$ & $\tau v_1$\\
			
			3& $\frac{3}{4}\tau^2-\delta-3$ & $(\frac{3}{4}\tau^2-\delta-3)u_1$
			& $(\frac{3}{4}\tau^2-\delta-3)v_1$\\[6pt]
			\hline

		\end{tabular}
	\end{table}

\end{proof}
With the notation, $\triangle(T^{n})=|u_n|^2-v_{n}^{2}$. Theorem \ref{34} only describes how the difference of squares $|u_n|^2-v_{n}^{2}$ changes, whereas Proposition \ref{35} describes separately how the two quantities $u_n$ and $v_n$ change. More importantly, Proposition \ref{35} provides a
finer classification of $T^n$ when $\triangle(T^{n})=0$. For example, suppose the original matrix $T$ satisfies
$$
\triangle(T)=0,\ \ u_{1}\neq 0.
$$
In this case, using only Theorem \ref{34}, one can only conclude that $\triangle(T^{n})=0$, and no further determination is possible. Proposition \ref{35}, however, gives
\begin{enumerate}
	\item If $q_{n}\neq 0$, then $u_{n}\neq 0$. The M\"{o}bius transformation $g_{T^{n}}$ corresponding to $T^n$ is compound parabolic.
	\item If $q_{n}= 0$, then $u_{n}=v_n=0$. At this point, we enter the classification of the simple type, and it is necessary to continue checking $b_{n}$ and $Tr(T^n)$.
\end{enumerate}
Proposition \ref{35} allows us to obtain $b_n-\overline{c_n}$ and $Re (a_n)-Re (d_n)$ by computing only a single real sequence $q_n$, without having to compute the entire $T^n$ separately for these two quantities. Note, however,
$$
q_n=0\ \Rightarrow\ b_n=\overline{c_n}
$$
does not imply $b_n=c_n=0$, nor does it imply $T^n=I_n$. Therefore, Proposition \ref{35} alone is still not sufficient to complete the classification of all types.

\begin{example}
	Let $T=\begin{bmatrix}
		-1+\boldsymbol{i}&\boldsymbol{i}\\
		\boldsymbol{i}&1+\boldsymbol{i}
	\end{bmatrix}$. Then $T\in U(1,1;\mathbb{H})$ and $g_{T}$ is compound parabolic. $T^2=\begin{bmatrix}
		-1-2\boldsymbol{i}&-2\\
		-2&-1+2\boldsymbol{i}
	\end{bmatrix}$, and $g_{T^{2}}$ is simple parabolic. In fact, if we set $N=\begin{bmatrix}
		\boldsymbol{i}&1\\
		1&-\boldsymbol{i}
	\end{bmatrix}$, then $T^{n}=\boldsymbol{i}^{n}(I_{2}+nN)$. Hence, $g_{T^{2n-1}}$ are compound parabolic and $g_{T^{2n}}$ are simple parabolic for integers $n\geq 1$.
\end{example}

\begin{example}
	Let $T=\frac{1}{4}\begin{bmatrix}
		5\boldsymbol{i}&3\boldsymbol{i}\\
		3\boldsymbol{i}&5\boldsymbol{i}
	\end{bmatrix}$. Then $T\in U(1,1;\mathbb{H})$ and $g_{T}$ is compound loxodromic. $T^2=-\frac{1}{8}\begin{bmatrix}
		17&15\\
		15&17
	\end{bmatrix}$, and $g_{T^{2}}$ is simple loxodromic.
\end{example}

\begin{example}
	Let $\zeta=\frac{1+\boldsymbol{i}}{\sqrt{2}}$ and $T=\begin{bmatrix}
		\zeta&0\\
		0&-\zeta
	\end{bmatrix}$. Then $T\in U(1,1;\mathbb{H})$ and $g_{T}$ is compound elliptic. $T^2=\begin{bmatrix}
		\boldsymbol{i}&0\\
		0&-\boldsymbol{i}
	\end{bmatrix}$, and $g_{T^{2}}$ is nontrivial simple elliptic.
\end{example}

\begin{example}
	Let $T=\begin{bmatrix}
		1&0\\
		0&\boldsymbol{i}
	\end{bmatrix}$. Then $T\in U(1,1;\mathbb{H})$ and $g_{T}$ is compound elliptic. $T^4=\begin{bmatrix}
		1&0\\
		0&1
	\end{bmatrix}=I_2$ and $g_{T^4}$ is trivial simple elliptic.
\end{example}
In what follows, we carry out a complete classification of the types of the M\"{o}bius transformations associated with $T^{n}$. For convenience, we write $\delta_{n}=\triangle(T^{n}),\ s_{n}=Tr(T^n)$.

\begin{theorem}[Classification of $g_{T^n}$]
	Let \(T\in U(1,1;\mathbb H)\), and let \(n\geq 1\) be an integer.
	Write
	\[
	T^n=
	\begin{bmatrix}
		a_n & b_n\\
		c_n & d_n
	\end{bmatrix},
	\qquad
	\delta_n=\Delta(T^n),
	\qquad
	s_n=\operatorname{Tr}(T^n).
	\]
	Then the type of \(g_{T^n}\) is completely determined as follows.
	\begin{enumerate}
		\item $g_{T^n}$ is compound elliptic if and only if
		\[
		\delta_n<0.
		\]
		
		\item $g_{T^n}$ is compound loxodromic if and only if
		\[
		\delta_n>0.
		\]
		
		\item $g_{T^n}$ is compound parabolic if and only if
		\[
		\delta_n=0
		\quad\text{and}\quad
		b_n\neq\overline{c_n}.
		\]
		
		\item $g_{T^n}$ is simple elliptic if and only if
		\[
		\delta_n=0
		\]
		and either
		\[
		b_n=\overline{c_n}=0,
		\]
		or
		\[
		b_n=\overline{c_n}\neq 0
		\quad\text{and}\quad
		s_n^2<16.
		\]
		
		\item $g_{T^n}$ is simple parabolic if and only if
		\[
		\delta_n=0,\qquad
		b_n=\overline{c_n}\neq 0,
		\qquad
		s_n^2=16.
		\]
		
		\item $g_{T^n}$ is simple loxodromic if and only if
		\[
		\delta_n=0,\qquad
		b_n=\overline{c_n}\neq 0,
		\qquad
		s_n^2>16.
		\]
	\end{enumerate}
	
\end{theorem}
\begin{proof}
	Fix an arbitrary integer \(n\geq 1\), and write
	\[
	T^n=
	\begin{bmatrix}
		a_n & b_n\\
		c_n & d_n
	\end{bmatrix}.
	\]
	
	Recall that
	\[
	\delta_n=\Delta(T^n)
	=
	\left|b_n-\overline{c_n}\right|^2
	-
	\left(\operatorname{Re}(a_n)
	-\operatorname{Re}(d_n)\right)^2
	\]
	and
	\[
	s_n=\operatorname{Tr}(T^n)
	=
	2\left(\operatorname{Re}(a_n)
	+\operatorname{Re}(d_n)\right).
	\]
	We distinguish the following three cases.
	
	\medskip
	\noindent
	\textbf{Case 1:}
	$b_n\neq\overline{c_n}$.
	
	By the classification recalled in Section~1, the type of
	$g_{T^n}$ is determined by the sign of $\delta_n$. More precisely,
	\[
	\begin{aligned}
		\delta_n<0
		&\iff g_{T^n}\text{ is compound elliptic},\\
		\delta_n=0
		&\iff g_{T^n}\text{ is compound parabolic},\\
		\delta_n>0
		&\iff g_{T^n}\text{ is compound loxodromic}.
	\end{aligned}
	\]
	This proves assertions \textup{(1)--(3)} in the case
	$b_n\neq\overline{c_n}$.
	
	\medskip
	\noindent
	\textbf{Case 2:}
	$b_n=\overline{c_n}=0$.
	
	In this case,
	\[
	\delta_n
	=
	-\left(\operatorname{Re}(a_n)
	-\operatorname{Re}(d_n)\right)^2
	\leq 0.
	\]
	If $\delta_n<0$, then
	\[
	\operatorname{Re}(a_n)\neq\operatorname{Re}(d_n),
	\]
	and the classification in Section~1 shows that $g_{T^n}$ is
	compound elliptic.
	
	If $\delta_n=0$, then
	\[
	\operatorname{Re}(a_n)=\operatorname{Re}(d_n),
	\]
	and hence $g_{T^n}$ is simple elliptic. This proves the first
	alternative in assertion \textup{(4)}.
	
	\medskip
	\noindent
	\textbf{Case 3:}
	$b_n=\overline{c_n}\neq0$.
	
	Since $T^n\in U(1,1;\mathbb H)$, we have
	\[
	a_n\overline{c_n}=b_n\overline{d_n}.
	\]
	Using $b_n=\overline{c_n}\neq0$, this gives
	\[
	a_nb_n=b_n\overline{d_n},
	\]
	and therefore
	\[
	b_n^{-1}a_nb_n=\overline{d_n}.
	\]
	Similar quaternions have the same real part, so
	\[
	\operatorname{Re}(a_n)
	=
	\operatorname{Re}(\overline{d_n})
	=
	\operatorname{Re}(d_n).
	\]
	Consequently,
	\[
	\delta_n=0
	\]
	and
	\[
	s_n
	=
	2\left(\operatorname{Re}(a_n)
	+\operatorname{Re}(d_n)\right)
	=
	4\operatorname{Re}(d_n).
	\]
	It follows that
	\[
	s_n^2=16\left(\operatorname{Re}(d_n)\right)^2.
	\]
	Hence
	\[
	\begin{aligned}
		s_n^2<16
		&\iff \left(\operatorname{Re}(d_n)\right)^2<1,\\
		s_n^2=16
		&\iff \left(\operatorname{Re}(d_n)\right)^2=1,\\
		s_n^2>16
		&\iff \left(\operatorname{Re}(d_n)\right)^2>1.
	\end{aligned}
	\]
	By the classification in Section~1, these three alternatives are
	equivalent, respectively, to $g_{T^n}$ being simple elliptic,
	simple parabolic, and simple loxodromic. This proves assertions
	\textup{(4)--(6)}.
	
	Since \(n\geq 1\) was arbitrary, the classification holds for every
	positive integer \(n\). The above discussion covers all cases in the classification. Therefore, the six conditions in the statement give
	a complete classification of $g_{T^n}$.
\end{proof}

We present code for classifying $g_{T^n}$, available at \url{https://github.com/HeJiaxing-hjx/U-1-1-H-_classification.git}.

\section*{Declarations}

\noindent \textbf{Ethics approval}

\noindent Not applicable.

\noindent \textbf{Conflict of interest/Competing interests}

\noindent The authors declare that there is no conflict of interest or competing interest.

\noindent \textbf{Funding}

\noindent This research received no external funding.

\noindent \textbf{Authors' contributions}

\noindent All authors contributed equally to this work.

\noindent \textbf{Availability of data and materials}

\noindent Our code is available at \url{https://github.com/HeJiaxing-hjx/U-1-1-H-_classification.git}. 

\bibliographystyle{plain}
\bibliography{reference}

\end{document}